\documentclass[11pt,reqno,palentino]{amsart}
\usepackage{mypack}
\usepackage{amscd}
\usepackage{pb-diagram}

\usepackage[backgroundcolor=white,linecolor=gray]{todonotes}%Adds \todo, which is nice.

\begin{document}

\title[Functoriality in Morse-Conley-Floer homology]{}

\author[T.O. Rot and R.C.A.M. Vandervorst]{}

\maketitle
\noindent {\huge {\bf  Functoriality and duality in Morse-Conley-Floer homology}} 
\vskip.8cm

\noindent 
T.O. Rot\footnote{Thomas Rot supported by NWO grant 613.001.001.
Email: thomas.rot@uni-koeln.de. } and R.C.A.M. Vandervorst\footnote{Email: vdvorst@few.vu.nl.}%\footnote{\today} 
\vskip.3cm

%\noindent {\it  Department of Mathematics, VU University Amsterdam, The Netherlands.}
\noindent{\it \Small Department of Mathematics, VU University Amsterdam, De Boelelaan 1081a, 1081 HV Amsterdam, the Netherlands.}
\vskip.8cm

\noindent{\begin{center}{\it\large Dedicated to Andrzej Granas.}\end{center}}

\vskip1cm
\begin{sloppypar}

\noindent {\bf Abstract.}
In~\cite{rotvandervorst} a homology theory --Morse-Conley-Floer homology-- for isolated invariant sets of arbitrary flows on finite dimensional manifolds is developed. In this paper we investigate functoriality and duality of this homology theory. As a preliminary we investigate functoriality in Morse homology. Functoriality for Morse homology of closed manifolds is known~\cite{abbondandoloschwarz, aizenbudzapolski,audindamian, kronheimermrowka, schwarz}, but the proofs use isomorphisms to other homology theories. We give direct proofs by analyzing appropriate moduli spaces. The notions of isolated map and flow map allow the results to generalize to local Morse homology and Morse-Conley-Floer homology. We prove Poincar\'e type duality statements for local Morse homology and Morse-Conley-Floer homology. \\

\noindent {\em AMS Subject Class:} 37B30, 37C10, 58E05


\textwidth=12.5cm
\vskip1cm

\graphicspath{{pics/}}
\section{Introduction}

In this paper we address functoriality and duality properties of Morse homology, local Morse homology and Morse-Conley-Floer homology. The functoriality of Morse homology on closed manifolds is known~\cite{abbondandoloschwarz, aizenbudzapolski,audindamian, kronheimermrowka, schwarz}, however no proofs are given through the analysis of moduli spaces. This analysis is done in Sections~\ref{sec:chainmap} through~\ref{sec:functoriality}. These sections are of independent interest to the rest of the paper. In Section~\ref{sec:isolation} we discuss isolation properties of maps, which are important for the functoriality in local Morse homology and Morse-Conley-Floer homology. This functoriality is discussed in Sections~\ref{sec:functorialitylocalmorse} and~\ref{sec:functorialitymorseconleyfloer}. In Section~\ref{sec:duality} we discuss Poincar\'e duality in these homology theories. Finally in Appendix~\ref{sec:generic} we prove that the transverse maps that are crucial for defining the induced maps in Morse homology are generic. Below is a detailed description of the results in this paper.

\subsection{Morse homology and local Morse homology}
\label{sec:morselocalmorseintro}
A \emph{Morse datum} is a quadruple $\cQ^\alpha=(M^\alpha,f^\alpha,g^\alpha,\co^\alpha)$, where $M^\alpha$ is  a choice of closed manifold and $(f^\alpha,g^\alpha,\co^\alpha)$ is a Morse-Smale triple on $M^\alpha$. Thus $f^\alpha$ is a Morse function on $M^\alpha$, and $g^\alpha$ is a metric such that the stable and unstable manifolds of $f^\alpha$ intersect transversely. Finally $\co^\alpha$ denotes a choice of orientations of the unstable manifolds. The Morse homology $HM_*(\cQ^\alpha)$ is defined as the homology of the chain complex of $C_*(\cQ^\alpha)$ which is freely generated by the critical points of $f^\alpha$ and graded by their index, with boundary operator $\partial_*(\cQ^\alpha)$ counting connecting orbits of critical points with Morse index difference of $1$ appropriately with sign, cf.~\cite{banyaga,schwarz,weber}. If $\cQ^\beta$ is another choice of Morse datum with $M^\beta=M^\alpha$ there is a canonical isomorphism $\Phi_*^{\beta\alpha}:HM_*(\cQ^\alpha)\rightarrow HM_*(\cQ^\beta)$. The canonical isomorphism is induced by continuation, i.e. a homotopy between the Morse data $\cQ^\alpha$ and $\cQ^\beta$, see for example~\cite{floerintersections,weber}. The Morse homology of the manifold $M=M^\alpha$ is defined by
\begin{equation}
\label{eq:morsehomology}
HM_*(M)=\varprojlim HM_*(\cQ^\alpha),
\end{equation}
where the inverse limit is taken over all Morse data with the canonical isomorphisms\footnote{Here and below we use the inverse limits to capture the full isomorphism class of the isomorphic structures $HM_*(\cQ^\alpha)$.}. The Morse homology $HM_*(M)$ is isomorphic to the singular homology $H_*(M^\alpha)$. 

Important in what follows is that this construction can also be carried out locally\footnote{Here we do not need to assume $M^\alpha$ is closed anymore.}. Recall the following definitions from Conley theory. A subset $S^\alpha\subset M^\alpha$ is called \emph{invariant} for a flow $\phi^\alpha$ if $\phi^\alpha(t,S^\alpha) = S^\alpha$, for all $t\in \mR$. A compact neighborhood $N^\alpha\subset M^\alpha$ is an \emph{isolating neighborhood} for $\phi^\alpha$ if $\Inv(N^\alpha,\phi^\alpha) \subset \Int(N^\alpha)$, where $$
\Inv\bigr(N^\alpha,\phi^\alpha\bigl) = \{x\in N^\alpha~|~\phi^\alpha(t,x) \in N^\alpha,~\forall t\in \mR\},
$$
is called the \emph{maximal invariant} set in $N^\alpha$. An invariant set $S^\alpha$ for which there exists an isolating neighborhood $N^\alpha$  with $S^\alpha=\Inv(N^\alpha,\phi^\alpha)$, is called an \emph{isolated invariant set}. A homotopy of flows is \emph{isolated} if $N^\alpha$ is an isolating neighborhood for each flow in the homotopy. Now suppose that $(f^\alpha,g^\alpha)$ is a Morse-Smale pair on $N^\alpha$, cf.~\cite[Definition~3.5]{rotvandervorst}. Thus $N^\alpha$ is an isolating neighborhood of the gradient flow  $\psi^\alpha$ of $(f^\alpha,g^\alpha)$ such that all critical points of $f^\alpha$ are non-degenerate on $N^\alpha$ and the local stable and unstable manifolds intersect transversely. Then $\cP^\alpha=(M^\alpha,f^\alpha,g^\alpha,N^\alpha,\co^\alpha)$, with $\co^\alpha$ a choice of orientation of the local unstable manifolds is a \emph{local Morse datum}. The local Morse homology $HM_*(\cP^\alpha)$ is defined by a similar counting procedure as Morse homology. However, now only critical points and connecting orbits that are contained in $N^\alpha$ are counted. If the gradient flows associated to two local Morse data $\cP^\alpha$ and $\cP^\beta$ are isolated homotopic\footnote{We use here the fact that two gradient flows are isolated homotopic through gradient flows if and only if they are isolated homotopic through arbitrary flows, see \cite[Proof of Proposition 5.3]{rotvandervorst}.} then there are canonical isomorphisms $\Phi_*^{\beta\alpha}:HM_*(\cP^\alpha)\rightarrow HM_*(\cP^\beta)$ induced by the continuation map. 

Local Morse homology is \emph{not} a topological invariant for $N^\alpha$. It measures dynamical information of the gradient flow of $(f^\alpha,g^\alpha)$. There is still stability under continuation. Given any function $f$ and metric $g$ such that $N$ is an isolating neighborhood of the gradient flow, which we do not assume to be Morse-Smale, we define the local Morse homology of such a triple via
$$
HM_*(f,g,N)=\varprojlim HM_*(\cP^\alpha).
$$
The inverse limit is taken over all local Morse data $\cP^\alpha$ on $N^\alpha=N$ whose gradient flow is isolated homotopic to the gradient flow $\psi$ of $(f,g)$ on $N$. 

If $M$ is a compact manifold, and we take $N=M$, the local Morse homology does not depend on the function and the metric anymore, as all gradient flows are isolated homotopic to each other. In this case the local Morse homology recovers the Morse homology defined in Equation~\bref{eq:morsehomology}. 

Results on Morse homology on compact manifolds with boundary fall in this framework. Let $M$ be a compact manifold with boundary. Assuming that the gradient of a Morse function $f$ is not tangent to the boundary, we can endow the boundary components with collars to obtain a manifold $\widetilde M$ without boundary. The function and metric extend to $\widetilde M$. Then $M\subset \widetilde M$ is an isolating neighborhood of the gradient flow. One can compute this Morse homology in terms of the singular homology of $M$ as $HM_*(f,g,M)\cong H_*(M,\partial M_-)$, where $\partial M_-$ is the union of the boundary components where the gradient of $f$ points outwards. For more details we refer to \cite[Section 1.1]{rotvandervorst}.

\subsection{Morse-Conley-Floer homology}

We recall the definition of Morse-Conley-Floer homology, cf.~\cite{rotvandervorst}, from a slightly different viewpoint. Let $\phi$ be a flow on $M$ and $S$ an isolated invariant set of the flow. A Lyapunov function $f^\alpha_{\phi}$ for $(S,\phi)$ is a function that is constant on $S$ and satisfies $\frac{d}{dt}\bigr|_{t=0}f^\alpha_{\phi}(\phi(t,p))<0$ on $N^\alpha\setminus S$ where $N^\alpha$ an isolating neighborhood for $S^\alpha$. Lyapunov functions always exist for a given isolated invariant set, cf.\ \cite[Proposition 2.6]{rotvandervorst}. We can compute the local Morse homology of a Lyapunov function with respect to the choice of a metric $e^\alpha$ and $N^\alpha$. 

Given another Lyapunov function $f^\beta_\phi$ which satisfies the Lyapunov property on $N^\beta$ and metric $e^\beta$, there is a canonical isomorphism $$\Phi_*^{\beta\alpha}:HM_*(f^\alpha_{\phi},e^\alpha,N^\alpha)\rightarrow HM_*(f^\beta_{\phi},e^\beta,N^\beta),$$ 
induced by continuation, cf.\ \cite[Theorems 4.7 and 4.8]{rotvandervorst}. The Morse-Conley-Floer homology of $(S,\phi)$ is then defined as the inverse limit over all such local Morse homologies
$$
HI_*(S,\phi)=\varprojlim HM_*(f^\alpha_{\phi},e^\alpha,N^\alpha).
$$
Morse-Conley-Floer homology is the local Morse homology of the Lyapunov functions. This definition is equivalent to 
$$
HI_*(S,\phi)=\varprojlim HM_*(\cR^\alpha),
$$
which runs over all Morse-Conley-Floer data $\cR^\alpha=(M,f^\alpha,g^\alpha,N^\alpha,\co^\alpha)$, where $(f^\alpha,g^\alpha,\co^\alpha)$ is a Morse-Smale triple on $N^\alpha$ whose gradient flow is isolated homotopic to the gradient flow of a Lyapunov function $f_\phi^\alpha$ on $N^\alpha$ for some metric $e^\alpha$, with respect to canonical isomorphisms $\Phi^{\beta\alpha}:HM_*(\cR^\alpha)\rightarrow HM_*(\cR^\beta)$. This was the viewpoint of~\cite{rotvandervorst}.

\subsection{Functoriality in Morse homology}

\label{sec:introfunctorialitymorse}
In Sections~\ref{sec:chainmap} through~\ref{sec:functoriality} we study functoriality in Morse homology on closed manifolds. Induced maps between Morse homologies are defined by counting appropriate intersections for \emph{transverse maps}. 

\begin{definition} 
\label{defi:transverse}
Let $h^{\beta\alpha}:M^\alpha\rightarrow M^\beta$ be a smooth map. We say that $h^{\beta\alpha}$ is \emph{transverse (with respect to Morse data $\cQ^\alpha$ and $\cQ^\beta$)}, if for all $x\in \crit f^\alpha$ and $y\in \crit f^\beta$, we have that
$$
h^{\beta\alpha}\bigr|_{W^u(x)}\pitchfork W^s(y). 
$$
The set of transverse maps is denoted by $\scrT(\cQ^\alpha,\cQ^\beta)$. We write $W_{h^{\beta\alpha}}(x,y)=W^u(x)\cap (h^{\beta\alpha})^{-1}(W^s(y))$ for the moduli spaces.
\end{definition}

Given $\cQ^\alpha$ and $\cQ^\beta$, the set of transverse maps $\scrT(\cQ^\alpha,\cQ^\beta)$ is generic, cf.~Theorem~\ref{thm:genericity}. The index of a critical point $x \in \crit f^\alpha$ is denoted $|x|$. The transversality assumption ensures that $W_{h^{\beta\alpha}}(x,y)$ is an oriented manifold of dimension $|x|-|y|$, cf.~Proposition~\ref{prop:manifolds}. Hence, for $|x|=|y|$ we can compute the oriented intersection number $n_{h^{\beta\alpha}}(x,y)$ and define an induced map $h^{\beta\alpha}_*:C_*(\cQ^\alpha)\rightarrow C_*(\cQ^\beta)$ via 
$$h^{\beta\alpha}_*(x)=\sum_{|y|=|x|} n_{h^{\beta\alpha}}(x,y)y.$$ 
In Sections~\ref{sec:chainmap} to~\ref{sec:functoriality} we show that this defines a chain map, that homotopic maps induce chain homotopic maps, and the composition of the induced maps is chain homotopic to the induced map of the composition. This implies that the induced map $h_*^{\beta\alpha}$ descents to a map
$$
h_*^{\beta\alpha}:HM_*(M^\alpha)\rightarrow HM_*(M^\beta),
$$
 between the Morse homologies via counting, which is functorial % $h^{\gamma\beta}_*h^{\beta\alpha}_*=(h^{\gamma\beta}\circ h^{\beta\alpha})_*$ with $\id^{\alpha\alpha}_*=\id$ 
and which does not depend on the homotopy class of the map $h^{\beta\alpha}$. The homotopy invariance and density of transverse maps allows for an extension to all smooth maps. %Therefore
\begin{theorem}
Morse homology is a functor $\HM_*:\MAN\rightarrow \GA$ between the category of smooth manifolds  to the category of graded abelian groups. The functor sends homotopic maps to the same map between the homology groups. 
\end{theorem}

\subsection{Isolation properties of maps}

To study functoriality in local Morse homology and Morse-Conley-Floer homology isolation properties of \emph{maps} are crucial. We identify the notion of \emph{isolated map} and \emph{isolated homotopy} in Definition~\ref{defi:isolatedmap}. Isolated maps are open in the compact-open topology, cf.~Proposition~\ref{prop:isolatedmapsareopen}, but are not necessarily functorial: the composition of two isolated maps need not be isolated.  % The composition of two isolated maps need not be isolated, but the property is open in the compact-open topology, cf.~Proposition~\ref{prop:isolationisopen}. 
An important class of maps that is isolated and form a category are flow maps. Flow maps were introduced by McCord~\cite{mccord} to study functoriality in Conley index theory. 
\begin{definition}
A smooth map $h^{\beta\alpha}:M^\alpha\rightarrow M^\beta$ between manifolds equipped with flows $\phi^\alpha$ and $\phi^\beta$, is a \emph{flow map} if it is proper and equivariant. Thus preimages of compact sets are compact and
$$h^{\beta\alpha}(\phi^\alpha(t,p))=\phi^\beta(t,h^{\beta\alpha}(p)),\quad \text{ for all }t\in \mR \quad\text{and  }p\in M^\alpha.
$$
\label{defi:flowmap}
\end{definition}

The isolation properties of flow maps are given in Proposition~\ref{prop:flowmaps}. If $h^{\beta\alpha}$ is a flow map, and $N^\beta$ is an isolating neighborhood, then $N^\alpha=(h^{\beta\alpha})^{-1}(N^\beta)$ is an isolating neighborhood. Moreover $h^{\beta\alpha}$ is isolated with respect to these isolating neighborhoods. Similar statments hold for compositions of flow maps.

\subsection{Functoriality in Local Morse homology}

In Section~\ref{sec:functorialitylocalmorse} we define induced maps in local Morse homology. Due to the local nature of the homology, not all maps are admissible and the notion of isolated map becomes crucial. The maps are computed by the same counting procedure, but now done locally. We sum up the functorial properties of local Morse homology from Propositions~\ref{prop:localmorseinduced} and~\ref{prop:localmorsefunctorial}:
\begin{enumerate}
\item[(i)] An isolated transverse map induces a chain map, hence descents to a map between the local Morse homologies.
\item[(ii)] An isolated homotopy between transverse maps induces a chain homotopic map.
\item[(iii)] If $h^{\gamma\beta}$ and $h^{\beta\alpha}$ and $h^{\gamma\beta}\circ h^{\beta\alpha}$ are transverse maps such that $h^{\gamma\beta}\circ \psi^\beta_R \circ h^{\beta\alpha}$ is an isolated map for all $R>0$, then $(h^{\gamma\beta}\circ h^{\beta\alpha})_*$ and $h^{\gamma\beta}_* h^{\beta\alpha}_*$ are chain homotopic.
\end{enumerate}

Consider the following category of isolated invariant sets of gradient flows.

\begin{definition}
The category of \emph{isolated invariant sets of gradient flows}, denoted by $\IGF$, has as objects the quadruples $(M,f,g,N)$ consisting of smooth functions $f$ on $M$ and  metrics $g$ such that the sets $N$ are an isolating neighborhoods for associated the gradient flows. A morphism $h^{\beta\alpha}:(M^\alpha,f^\alpha,g^\alpha,N^\alpha)\rightarrow (M^\beta,f^\beta,g^\beta,N^\beta)$, is a map that is isolated homotopic to a flow map $\widetilde h^{\beta\alpha}$ with $N^\alpha=(\widetilde h^{\beta\alpha})^{-1}(N^\beta)$.
\end{definition}

These morphisms are then perturbed to transverse maps, from which the induced map can be computed, cf.~Propositions~\ref{prop:degenerateinduced} and~\ref{prop:degeneratefunctoriality}. 

\begin{theorem}
Local Morse homology is a covariant functor $HM_*:\IGF\rightarrow \GA$. The functor is constant on isolated homotopy classes of maps. 
\label{thm:localmorsefunctoriality}
\end{theorem}

The local Morse homology functor generalizes the Morse homology functor. Any map $h^{\beta\alpha}:M^\alpha\rightarrow M^\beta$ between closed manifolds equipped with flows is isolated homotopic to a flow map with $M^\alpha=N^\alpha=(h^{\beta\alpha})^{-1}(N^\beta)=(h^{\beta\alpha})^{-1}(M^\beta)$, since the flows on both manifolds are isolated homotopic to the constant flow, and any map is equivariant with respect to constant flows. 

\subsection{Functoriality in Morse-Conley-Floer homology}

Functoriality in Morse-Conley-Floer homology now follows from the results on local Morse homology, cf.~Theorem~\ref{thm:morseconleyfloerfunctoriality} by establishing appropriate isolated homotopies. 
%We have the following category of isolated invariant sets
\begin{definition}
The \emph{category of isolated invariant sets} $\IS$ has as objects triples $(M,\phi,S)$ of a manifold, flow and isolated invariant set. A morphism $h^{\beta\alpha}:(M^\alpha,\phi^\alpha,S^\alpha)\rightarrow (M^\beta,\phi^\beta,S^\beta)$ is a map that is isolated homotopic for some choice of isolating neighborhoods $N^\alpha,N^\beta$ to a flow map $\widetilde h^{\beta\alpha}$ such that $N^\alpha=(\widetilde h^{\beta\alpha})^{-1}(N^\beta)$. 
\end{definition}

%The main theorem of this chapter is.

\begin{theorem}
Morse-Conley-Floer homology is a covariant functor $HI_*:\IS\rightarrow \GA$. The functor is constant on the isolated homotopy classes of maps and flows. 
\end{theorem}

\subsection{Duality in Morse-Conley-Floer Homology}
\label{subsec:dual}
In Section~\ref{sec:duality} we prove a Poincar\'e type duality theorem for Morse-Conley-Floer (co)homology. The theorem expresses a duality between the Morse-Conley-Floer homology of an orientable isolated invariant set $S$ of a flow and the Morse-Conley-Floer cohomology of $S$ seen as an isolated invariant set of the reverse flow defined by $\phi^{-1}(t,x)=\phi(-t,x)$. The following duality statement resembles an analogous theorem for the Conley index, due to McCord~\cite{McCord:1992vy}.

\begin{theorem}
Let $S$ be an oriented isolated invariant set of a flow $\phi$. Then
there are Poincar\'e duality isomorphisms
$$
PD_k:HI_k(S,\phi)\rightarrow HI^{m-k}(S,\phi^{-1}).
$$
\end{theorem}

\subsection*{Acknowledgments}

We would like to express our gratitude to Alberto Abbondandolo, Peter Albers, Patrick Hafkenscheid and Federica Pasquotto for helpful discussions concerning this work.

\section{Chain maps in Morse homology on closed manifolds}

\label{sec:chainmap}

In Sections~\ref{sec:chainmap} through~\ref{sec:functoriality}, which discuss functoriality for Morse homology we assume that the base manifolds are closed. 

\subsection{The moduli space $W_{h^{\beta\alpha}}(x,y)$}
For a transverse map, see Definition~\ref{defi:transverse}, the moduli spaces $W_{h^{\beta\alpha}}(x,y)=W^u(x)\cap (h^{\beta\alpha})^{-1}(W^s(y))$ are smooth oriented manifolds. 

\begin{proposition}
Let $h^{\beta\alpha}\in \scrT(\cQ^\alpha,\cQ^\beta)$. For all $x\in\crit f^\alpha$ and $y\in \crit f^\beta$, the space $W_{h^{\beta\alpha}}(x,y)$ is an oriented submanifold of dimension $|x|-|y|$.
\label{prop:manifolds}
\end{proposition}

\begin{proof}[Proof of Proposition~\ref{prop:manifolds}.]
Because $h^{\beta\alpha}$ restricted to $W^u(x)$ is transverse to $W^s(y)$, Theorem 3.3. of~\cite[Theorem 3.3, page 22]{Hirsch} implies that $W_{h^{\beta\alpha}}(x,y)=\left(h^{\beta\alpha}\bigr|_{W^u(x)}\right)^{-1}(W^s(y))$ is an oriented submanifold of $W^u(x)$. The orientation is induced by the exact sequence of vector bundles\footnote{We employ the fiber first convention in the orientation of exact sequences of vector bundles, cf.~\cite{Hirsch,kronheimermrowka}}
$$
\xymatrix{0\ar[r]&TW_{h^{\beta\alpha}}(x,y)\ar[r]&TW^u(x)\ar[r]^{dh^{\beta\alpha}}&NW^s(y)\ar[r]&0},
$$
where the latter is the normal bundle of $W^s(y)$. The normal bundle of $W^s(y)$ is oriented, because $W^s(y)$ is contractible, and $N_yW^s(y)\cong T_yW^u(y)$ is oriented by the choice $\co^\beta$. The codimension of $W_{h^{\beta\alpha}}(x,y)$ in $W^u(x)$ equals the codimension of $W^s(y)$ in $M^{\beta}$. Thus
$$
|x|-\dim W_{h^{\beta\alpha}}(x,y)=\codim W^s(y)=|y|,
$$
from which the proposition follows.
\end{proof}

If $|x|=|y|$ the space $W_{h^{\beta\alpha}}(x,y)$ is zero dimensional and consist of a finite number of points carrying orientation signs $\pm 1$. Set $n_{h^{\beta\alpha}}(x,y)$ as the sum of these. We define the degree zero map $h^{\beta\alpha}_*:C_*(\cQ^\alpha)\rightarrow C_*(\cQ^\beta)$ by the formula
$$
h^{\beta\alpha}_*(x)=\sum_{|y|=|x|}n_{h^{\beta\alpha}}(x,y)y.
$$
Through compactness and gluing analysis of the moduli space $W_{h^{\beta\alpha}}(x,y)$ and related moduli spaces we prove the following properties of the induced map
\begin{enumerate}
\item[(i)] The induced identity $\id_*^{\alpha\alpha}$ is the identity on chain level. 
\item[(ii)] We show that $h_*^{\beta\alpha}$ is a chain map, i.e.
$$
h_{k-1}^{\beta\alpha}\partial^\alpha_k=\partial^\beta_kh^{\beta\alpha}_{k},\quad\text{for all}\quad k.
$$
\item[(iii)] In Section~\ref{sec:homotopy} we study homotopy invariance. Suppose that $\cQ^\gamma,\cQ^\delta$ are other Morse data on $M^\alpha=M^\gamma$ and $M^\beta=M^\delta$. Suppose $h^{\delta\gamma}\in \scrT(\cQ^\gamma,\cQ^\delta)$ is homotopic to $h^{\beta\alpha}\in \scrT(\cQ^\alpha,\cQ^\beta)$, then $h_*^{\delta\gamma}\Phi^{\gamma\alpha}_*$ and $\Phi^{\delta\beta}_*h_*^{\beta\alpha}$ are chain homotopic. That is, there exists a degree $+1$ map $P_*^{\delta\alpha}:C_*(\cQ^\alpha)\rightarrow C_*(\cQ^\delta)$, such that
$$
\Phi_k^{\delta\beta}h^{\beta\alpha}_k-h^{\delta\gamma}_k\Phi_k^{\gamma\alpha}=-\partial^\delta_{k+1}P_k^{\delta\alpha}-P_{k-1}^{\delta\alpha}\partial_k^{\alpha},\quad\text{for all}\quad k.
$$
Here $\Phi^{\beta\alpha}_*$ denotes the isomorphism induced by continuation.
\item[(iv)] In Section~\ref{sec:functoriality} we study compositions. If $h^{\gamma\beta}\in \scrT(\cQ^\beta,\cQ^\gamma)$ is such that $h^{\gamma\beta}\circ h^{\beta\alpha}\in \scrT(\cQ^\alpha,\cQ^\gamma)$ then $h_*^{\gamma\beta}h_*^{\beta\alpha}$ and $\left(h^{\gamma\beta}\circ h^{\beta\alpha}\right)_*$ are chain homotopic, i.e.~there is a degree $+1$ map $P_*^{\gamma\alpha}:C_*(\cQ^\alpha)\rightarrow C_*(\cQ^\gamma)$, such that
$$
h_k^{\gamma\beta} h_k^{\beta\alpha}-\left(h^{\gamma\beta}\circ h^{\beta\alpha}\right)_k=P_{k-1}^{\gamma\alpha}\partial^\alpha_k+\partial_{k+1}^\gamma P_k^{\gamma\alpha},\quad\text{for all}\quad k.
$$
\end{enumerate}

\subsection{Compactness of $W_{h^{\beta\alpha}}(x,y')$ with $|x|=|y'|+1$}

The chain map property (ii) holds by the compactness properties of the moduli space $W_{h^{\beta\alpha}}(x,y')$ with $|x|=|y'|+1$. This space is a one dimensional manifold, but it is not necessarily compact. The non-compactness is due to breaking of orbits in the domain and in the codomain, which is the content of Proposition~\ref{prop:compactness}. In Figure~\ref{fig:compactness1} we depicted this breaking process in the domain. Recall that we denote by $W(x,y)=W^u(x)\cap W^s(y)$ the space of parameterized orbits, and by $M(x,y)=W(x,y)/\mR$ the space of unparameterized orbits. The points where breaking can occur in the domain are counted by $M(x,y)\times W_{h^{\beta\alpha}}(y,y')$, with $|y|=|y'|$ and the points where breaking occurs in the codomain are counted by $W_{h^{\beta\alpha}}(x,x')\times M(x',y')$, with $|x|=|x'|$. The space $W_{h^{\beta\alpha}}(x,y')$ can be compactified by gluing in these broken orbits, cf.~Proposition~\ref{prop:gluing2}. Proposition~\ref{prop:compactification} then states that the resulting object is a one dimensional compact manifold with boundary. By counting the boundary components appropriately with sign the chain map property is obtained, cf.~Proposition~\ref{prop:chain}.

\begin{figure}
\def\svgwidth{.9\textwidth}\begin{center}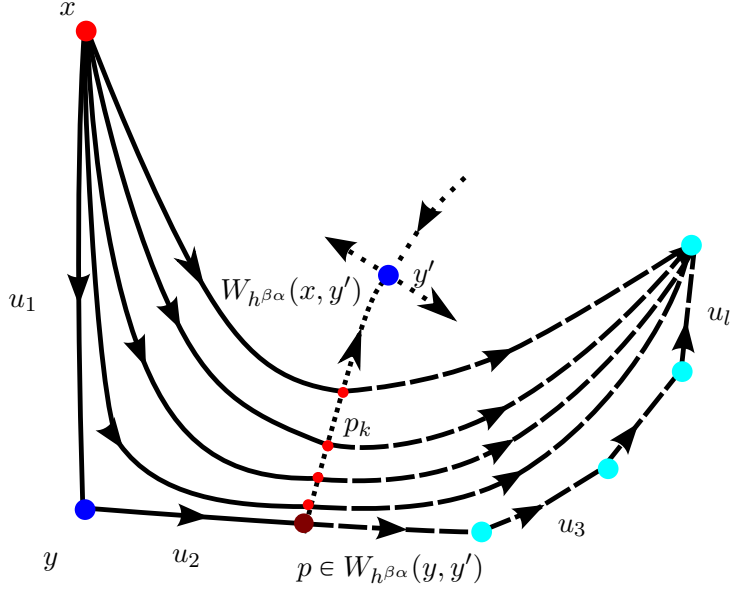\end{center}
\caption{One of the possible compactness failures of $W_{h^{\beta\alpha}}(x,y')$ with $|x|=|y'|+1$. The sequence $p_k\in W_{h^{\beta\alpha}}(x,y')$ converges to a point $p\in W_{h^{\beta\alpha}}(y,y')$ with $|y|=|y'|$. The orbits through $p_k$ break to $(u_1,\ldots,u_l)$ with $u_1\in M(x,y)$.}
\label{fig:compactness1}
\end{figure}
Since $M^{\alpha}$ is assumed to be compact, and $W_{h^{\beta\alpha}}(x,y')\subset M^\alpha$, any sequence $p_k\in W_{h^{\beta\alpha}}(x,y')$ has a subsequence converging to some point $p\in M^{\alpha}$. In the next proposition, see also Figure~\ref{fig:compactness1}, it is stated to which points such sequences converge. 

\begin{proposition}[Compactness]
Let $h^{\beta\alpha}\in \scrT(\cQ^\alpha,\cQ^\beta)$, $x\in\crit f^\alpha$ and $y'\in \crit f^\beta$ with $|x|=|y'|+1$. Let $p_k\in W_{h^{\beta\alpha}}(x,y')$ be a sequence such that $p_k\rightarrow p$ in $M^{\alpha}$. Then either one of the following is true
\begin{itemize}
\item[(i)] $p\in W_{h^{\beta\alpha}}(x,y')$.
\item[(ii)] There exists $y\in\crit f^\alpha$, with $|y|=|y'|$ such that $p\in W_{h^{\beta\alpha}}(y,y')$. The orbits through $p_k$ break to orbits $(u_1,\ldots,u_l)$ as $k\rightarrow \infty$ with $u_1\in M(x,y)$.
\item[(iii)]There exists $x'\in\crit f^\beta$, with $|x'|=|x|$ such that $p\in W_{h^{\beta\alpha}}(x,x')$. The orbits through $h^{\beta\alpha}(p_k)$ break to the orbits $(u_1',\ldots,u_{l}')$ as $k\rightarrow \infty$ with $u_{l}'\in M(x',y')$.
\end{itemize} 
\label{prop:compactness}
\end{proposition}
\begin{proof}
The spaces $W_{h^{\beta\alpha}}(x,y')$, $W_{h^{\beta\alpha}}(y,y')$ and $W_{h^{\beta\alpha}}(x,x')$ are disjoint hence the three possibilities cannot occur at the same time. We choose a subsequence such that $p_k\in W(x,b)$ for some fixed $b\in \crit f^\alpha$, which is possible by the fact that there are only a finite number of such spaces by compactness. The Broken Orbit Lemma, see for example~\cite[Lemma 2.5]{braam}, states that $p\in W(y,a)$ for some $y,a\in \crit f^\alpha$ with $|y|\leq |x|$, and equality if and only if $x=y$. Similarly, choosing a subsubsequence if necessary, we also assume that $h^{\beta\alpha}(p_k)\in W(a',y')$, for some fixed $a'\in\crit f^\beta$. Then $h^{\beta\alpha}(p)\in W(b',x')$, for some $b',x'\in \crit f^\beta$ with $|x'|\geq |y'|$, with equality if and only if $x'=y'$. 

Now if $p\not \in W^u(x)$ and $h^{\beta\alpha}(p)\not \in W^s(y')$, then we have the impossibility that $p\in W_{h^{\beta\alpha}}(y,x')$, with $|y|<|x|$ and $|x'|>|y'|$. Since by transversality
$$
\dim W_{h^{\beta\alpha}}(y,x')=|y|-|x'|\leq (|x|-1)-(|y'|-1)\leq-1.
$$
Assuming that $p\not\in W_{h^{\beta\alpha}}(x,y')$, only two possibilities remain. If $p\in W_{h^{\beta\alpha}}(x,x')$, with $|x'|>|y|=|x|-1$, then from $\dim W_{h^{\beta\alpha}}(x,x')\geq 0$ it follows that $|x'|=|x|$. If $h^{\beta\alpha}(p)\in W_{h^{\beta\alpha}}(y,y')$, with $|y|<|x|$, then $\dim W_{h^{\beta\alpha}}(y,y')\geq 0$ gives that $|y'|=|y|$, see also~Figure~\ref{fig:compactness1}. The claim about the breaking orbits is the content of the Broken Orbit Lemma.
\end{proof}

The proposition generalizes to higher index difference moduli spaces. We do not need this, and this would clutter the notation without a significant gain. 

\subsection{Gluing the ends of $W_{h^{\beta\alpha}}(x,y')$, with $|x|=|y'|+1$}

We compactify $W_{h^{\beta\alpha}}(x,y')$ by gluing in the broken orbits described in  Proposition~\ref{prop:compactness}. The following lemma is the technical heart of the standard gluing construction in Morse homology, see also Figure~\ref{fig:gluing}. We single this out because we have to construct several gluing maps, which all use this lemma. For a pair $(f,g)$ of a function and a metric we denote by $\psi$ its negative gradient flow. 
\begin{figure}
\def\svgwidth{.4\textwidth}\begin{center}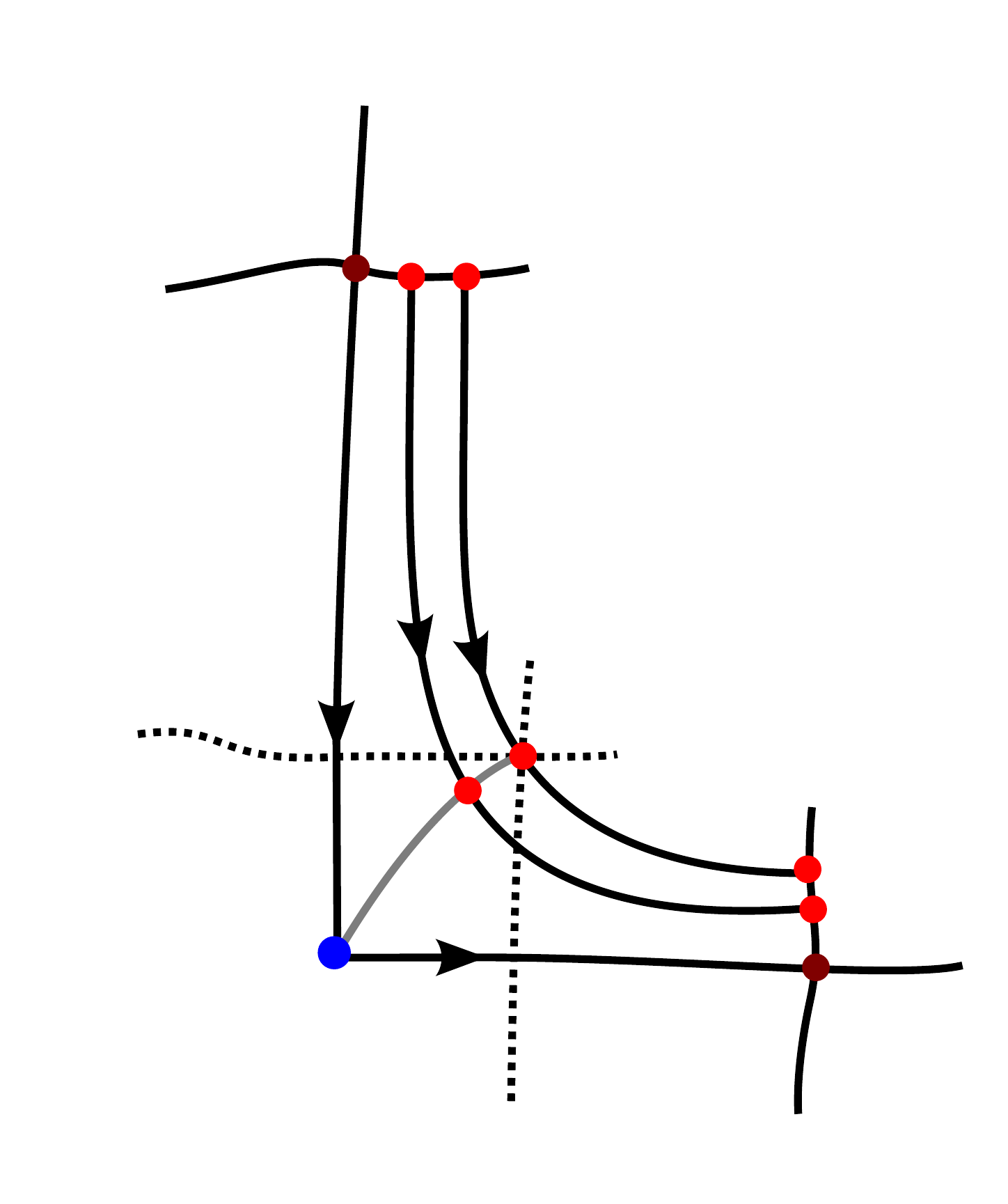\end{center}
\caption{The content of Lemma~\ref{lem:gluing} is depicted. Discs transverse to the stable and unstable manifolds must intersect if they are flowed in forwards and backwards time. This intersection point is used to define the gluing map. }
\label{fig:gluing}
\end{figure}
\begin{lemma}[Gluing] Let $f:M\rightarrow \mR$ be a Morse function, $g$ a metric, and $y\in\crit f$. Write $m=\dim M$. Suppose $D^{|y|},E^{m-|y|}$ are embedded discs of dimension $|y|$ and $m-|y|$ with $D^{|y|}\pitchfork W^s(y)$ and $E^{m-|y|}\pitchfork W^u(y)$. Assume that each intersection only consists of a single point and write $u\in D^{|y|}\cap W^s(y)$ and $v\in E^{m-|y|}\cap W^u(y)$. Then there exists an $R_0>0$ and an injective map $\rho:[R_0,\infty)\rightarrow M$, such that $g(\frac {d}{dR} \rho(R),-\nabla_g f)\not=0$, and $\psi(-R,\rho(R))\in D^{|y|}$, and $\psi(R,\rho(R))\in E^{m-|y|}$. We have the limits
$$
\lim_{R\rightarrow \infty} \rho(R)=y,\qquad \lim_{R\rightarrow \infty} \psi(-R,\rho(R))=u,\qquad \lim_{R\rightarrow \infty} \psi(R,\rho(R))=v.
$$
Finally there exists smaller discs $D'\subset D^{|y|}$ and $E'\subset E^{m-|y|}$ such that no orbit through $D'\setminus \bigcup_{R\in [R_0,\infty)}\psi(-R,\rho(R))$ intersects $E'$.
\label{lem:gluing}
\end{lemma}

\begin{proof}
We only sketch the proof, see also Figure~\ref{fig:gluing}. More details can be found in the proof of Theorem 3.9 in~\cite{weber}. Let $B^u\subset W^u(y)$, $B^s\subset W^s(y)$ be closed balls containing $y$. Write $D_R^{|y|}=\psi(R,D^{|y|})$ and $E_{-R}^{m-|y|}=\psi(-R,E^{m-|y|})$. Since the discs are transverse to the stable and unstable manifolds of $y$, the $\lambda$-Lemma, cf.~\cite[Chapter 2, Lemma 7.2]{palismelo}, gives that for all $t$ large, smaller discs  $D_R'\subset D_R^{|y|}$ and $E'_{-R}\subset E_{-R}^{m-|y|}$ are $\epsilon-C^1$ close to $B^u$ and $B^s$ respectively. It follows, through an application of the Banach Fixed Point Theorem, that there exist a $R_0>0$ such that $D'_R$ and $E_{-R}'$ intersect in a single point for each $R>R_0$ sufficiently large.  Set $\rho(R)=D'_R\cap E_{-R}'$. The properties of $\rho$ follow from the construction. 
\end{proof}

To prove that $\partial^2=0$ in Morse homology gluing maps are needed to compactify appropriate moduli spaces and these are constructed using the Morse-Smale condition and previous lemma as follows. Let $x,y$ and $z$ be critical points with $|x|=|y|-1=|z|-2$ and assume $M(x,y)$ and $M(y,z)$ are non-empty. Then the Morse-Smale condition gives that there exists a disc $D^{|y|}$ in $W^u(x)$ transverse to an orbit in $M(x,y)$. Because of transversality, one can also choose a disc $E^{m-|y|}$ in $W^s(z)$ transverse to an orbit in $M(y,z)$. The gluing map $\#:M(x,y)\times [R_0,\infty)\times M(y,z)\rightarrow M(x,z)$, is given by mapping $R$ to the orbit through $\rho(R)$ as in~Figure~\ref{fig:gluing}. We now use similar ideas to compactify $W_{h^{\beta\alpha}}(x,y')$ with $|x|=|y'|+1$. 

\begin{proposition}
Assume $h^{\beta\alpha}\in \scrT(\cQ^\alpha,\cQ^\beta)$. Then for critical points $x,y\in \crit f^\alpha$ and $x',y'\in \crit f^\beta$, with $|x|=|x'|=|y|+1=|y'|+1$, there exists $R_0>0$ and gluing embeddings
\begin{align*}
\#^1:&M(x,y)\times [R_0,\infty)\times W_{h^{\beta\alpha}}(y,y')\rightarrow W_{h^{\beta\alpha}}(x,y')\\
\#^2:&W_{h^{\beta\alpha}}(x,x')\times [R_0,\infty)\times M(x',y')\rightarrow W_{h^{\beta\alpha}}(x,y')
\end{align*}
Moreover, if $p_k\in W_{h^{\beta\alpha}}(x,y')$ converges to $p\in W_{h^{\beta\alpha}}(y,y')$, and the orbits through $p_k$ break as $(u_1,\ldots,u_l)$ with $u_1\in M(x,y)$ then $p_k$ is in the image of $\#^1$ for $k$ sufficiently large. Analogously, if $p_k\in W_{h^{\beta\alpha}}(x,y')$ converges to $p\in W_{h^{\beta\alpha}}(x,x')$, and the orbits through $h^{\beta\alpha}(p_k)$ break to $(u_1,\ldots,u_l)$ with $u_l'\in M(x',y)$, then $p_k$ is in the image of $\#^2$ for $k$ sufficiently large.
\label{prop:gluing2}
\end{proposition} 
\begin{proof}
Let $u \in W(x,y)$, and $v\in W_{h^{\beta\alpha}}(y,y')$. Since $(f^\alpha,g^\alpha)$ is a Morse-Smale pair, we can choose a disc $D^{|y|}$ in $W^u(x)$ through $u$ and transverse to $W^s(y)$. Analogously, because $h^{\beta\alpha}$ is transverse, we can find a disc $E^{m^\alpha-|y|}$ in $M^\alpha$ through $v$, and transverse to $W^u(y)$, with $h^{\beta\alpha}(E^{m^\alpha-|y|})\subset W^s(y')$. Lemma~\ref{lem:gluing} provides us with a map $\rho:[R_0,\infty)\rightarrow M^\alpha$. Denote by $\gamma_u\in M(x,y)$ the orbit through $u$, and set $\gamma_u\#^1_Rv=\psi^\alpha(R,\rho(R))$. Then $\gamma_u\#^1_Rv\in E^{m^\alpha-|y|}$ and $\psi^\alpha(-2R,\gamma_u\#^1_Rv)\in D^{|y|}\subset W^u(x)$, hence $\gamma_u\#^1_Rv\in W_{h^{\beta\alpha}}(x,y')$. The properties of the map $\rho$ in Lemma~\ref{lem:gluing} directly give the properties of $\#^1$. 

The construction of $\#^2$ is similar. Let $u\in W_{h^{\beta\alpha}}(x,x')$ and $v\in W(x',y')$. By transversality of $h^{\alpha\beta}$ we can choose a disc $D^{|x'|}\subset W^{u}(x)$ which $h^{\beta\alpha}$ maps bijectively to $h^{\beta\alpha}(D^{|x'|})$, and whose image $h^{\beta\alpha}(D)$ is transverse to $W^s(x')$. Choose a disc $E^{m^\beta-|x'|}\subset W^s(y')$ with $v\in E^{m^\beta-|x'|}$ which is transverse to $W^u(x')$. Now Lemma~\ref{lem:gluing} provides us with a map $\rho:[R_0,\infty)\rightarrow M^\beta$. Set $u\#^2_R\gamma_v=(h^{\beta\alpha})^{-1}(\psi^\beta(-R,\rho(R)))$, which is well defined since $h^{\beta\alpha}|_{D^{|x'|}}$ is a diffeomorphism onto $h^{\beta\alpha}(D^{|x'|})$. The properties of $\#^2$ follow.
\end{proof}

\subsection{Orientations.} After a choice of orientations of the unstable manifolds, the moduli spaces $W_{h^{\beta\alpha}}(x,y')$ and $W(x,y)$ and $M(x,y)$ carry induced orientations cf.~Proposition~\ref{prop:manifolds} and~\cite[Proposition 3.10]{weber}. We show that the gluing map $\#^1$ is compatible with the induced orientations, while $\#^2$ reverses the orientations. 

Consider the notation of the proof of Proposition~\ref{prop:gluing2}. Denote by $W(x,y)\Bigr|_u$ the connected component of $W(x,y)$ containing $u\in W(x,y)$, with similar notation for other moduli spaces. If $|x|=|y|+1$ the moduli space $W(x,y)$ is one dimensional and can be oriented by the negative gradient vector field. For $u\in W(x,y)$ write $[\dot u]$ for this induced orientation of $W(x,y)\Bigr|_u$. The orientation of $W(x,y)\Bigr|_u$ induced by the choice of $\co^\alpha$ is denoted $a[\dot u]$, with $a=\pm1$. Then we denote for the orientation of $W_{h^{\beta\alpha}}(y,y')$, induced by $\co^\alpha$ and $\co^\beta$, by $b$ for $b=\pm 1$. The gluing map $\#^1$ induces a map of orientations 
$$\sigma^1:\Or(W(x,y)\Bigr|_u)\times\Or(W_{h^{\beta\alpha}}(y,y') \Bigr|_v)\rightarrow \Or(W_{h^{\beta\alpha}}(x,y') \Bigr|_{\gamma_u\#^1_{R_0}v}),$$ 
via
$$
\sigma^1(a[\dot u],b)=ab\left[\frac{d}{dR}\Bigr|_{R=R_0} \gamma_u\#^1_Rv\right].
$$
Similarly the gluing map $\#^2$ induces a map of orientations $$\sigma^1:\Or(W_{h^{\beta\alpha}}(x,x')\Bigr|_u)\times \Or(W(x',y') \Bigr|_v)\rightarrow\Or(W_{h^{\beta\alpha}}(x,y')\Bigr|_{u\#^2_{R_0}\gamma_v})$$ 
given by
$$
\sigma^2(c,d[\dot v])=cd \left[ \frac{d}{dR}\Bigr|_{R=R_0}u\#^2_R\gamma_v\right],\qquad\text{with}\qquad c,d=\pm1.
$$

\begin{figure}
\def\svgwidth{.6\textwidth}\begin{center}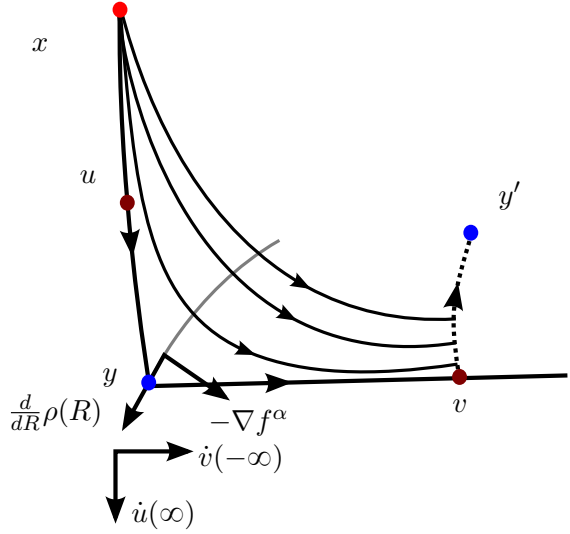\end{center}
\caption{Orientation issues due to breaking in the domain. The orientations $[\dot u(\infty), \dot v(-\infty)]$ and $\left[ \frac{d}{dR} \rho(R),-\nabla f^\alpha\right]$ agree, from which it follows that the map $\sigma^1$ is orientation preserving.}
\label{fig:orientation}
\end{figure}

\begin{proposition}Let the notation be as above. Then $\sigma^1$ preserves the orientation, and $\sigma^2$ reverses the orientation induced by $\co^\alpha$ and $\co^\beta$. 
\end{proposition}
\begin{proof}
We first treat the gluing map $\sigma^1$. By the transversality assumptions we have the following exact sequences of oriented vector spaces
\begin{gather}
\xymatrix{0\ar[r]&T_uW(x,y)\ar[r]&T_uW^u(x)\ar[r]&N_u W^s(y)\ar[r]&0},
\label{eq:or1}\\
\xymatrix{0\ar[r]&T_vW_{h^{\beta\alpha}}(y,y')\ar[r]&T_vW^u(y)\ar[r]^-{dh^{\beta\alpha}}&N_{h^{\beta\alpha}(v)} W^s(y')\ar[r]&0},
\label{eq:or2}\\
\xymatrix@C=.23cm{0\ar[r]&T_{\gamma_u\#^1_{R_0}v}W_{h^{\beta\alpha}}(x,y')\ar[r]&T_{\gamma_u\#^1_{R_0}v}W^u(x)\ar[rr]^-{dh^{\beta\alpha}}&&N_{h^{\beta\alpha}({\gamma_u\#^1_{R_0}v)}} W^s(y')\ar[r]&0}.
\label{eq:or3}
\end{gather}
The following isomorphisms of oriented vector spaces are induced by parallel transport and the fact that the stable and unstable manifolds are contractible
\begin{gather}
\label{eq:iso1}
T_uW^u(x)\cong T_{\gamma_u\#^1_{R_0}v}W^u(x),\\
\label{eq:iso2}
 N_{h^{\beta\alpha}(v)}W^s(y')\cong N_{h^{\beta\alpha}({\gamma_u\#^1_{R_0}v)}} W^s(y')\cong N_{y'}W^s(y')\cong T_{y'}W^u(y'),\\
N_uW^s(y)\cong N_yW^s(y)\cong T_yW^u(y)\cong T_vW^u(y).
\label{eq:iso3}
\end{gather}
Along with the identification of the normal bundle of the stable manifold with the tangent bundle of the unstable manifold. From the exact sequences \bref{eq:or1} and \bref{eq:or2} and isomorphisms \bref{eq:iso2} and~\bref{eq:iso3} it now follows that, as oriented vector spaces,
\begin{align*}
T_uW^u(x)&\cong T_uW(x,y)\oplus N_uW^s(y)\\
&\cong T_uW(x,y)\oplus T_vW_{h^{\beta\alpha}}(y,y')\oplus N_{h^{\beta\alpha}(v)}W^s(y')\\
&\cong T_uW(x,y)\oplus T_vW_{h^{\beta\alpha}}(y,y')\oplus T_{y'}W^u(y').
\end{align*}
Analogously, from~\bref{eq:or3},~\bref{eq:iso1} and~\bref{eq:iso2} we get
\begin{align*}
T_uW^u(x)\cong T_{\gamma_u\#^1_{R_0}v}W^u(x)&\cong T_{\gamma_u\#^1_{R_0}v}W_{h^{\beta\alpha}}(x,y')\oplus N_{h^{\beta\alpha}({\gamma_u\#^1_{R_0}v)}} W^s(y')\\
&\cong T_{\gamma_u\#^1_{R_0}v}W_{h^{\beta\alpha}}(x,y')\oplus T_{y'}W^u(y').
\end{align*}
Combining the last two formulas we see that
$$
T_{\gamma_u\#^1_{R_0}v}W_{h^{\beta\alpha}}(x,y')\cong T_uW(x,y)\oplus T_vW_{h^{\beta\alpha}}(y,y').
$$
The orientation on the right hand side is $ab[\dot u]$. The tangent vector to the orbit $\gamma_u$ through $u$ has a well defined limit as $t\rightarrow \infty$  which we denote by $[\dot u(\infty)]$, and similarly the tangent vector to $\gamma_v$ has a well defined limit as $t\rightarrow -\infty$, which we denote by $[\dot v(-\infty)]$, cf.~\cite[Theorem 3.11]{weber} and~\cite[Lemma B.5]{schwarz}. We want to compare the orientation $[\dot u]$ to $\left[\frac{d}{dR}\Bigr|_{R=R_0} \gamma_u\#^1_Rv\right]$, and are able to do this as follows. We can flow $W_{h^{\beta\alpha}}(x,y')\Bigr|_{\gamma_u\#^1_{R_0}v}$ with the gradient flow, cf.~Figure~\ref{fig:orientation}, which we denote by $W _{h^{\beta\alpha}} (x,y')\Bigr|_{\gamma_u\#^1_{R_0}v}\times \mR)$. Note that $y$ is in the closure of this space. Then the orientations $[\dot u(\infty), \dot v(-\infty)]$ and $\left[ \frac{d}{dR}\Bigr|_{R=R_0} \gamma_u\#^1_Rv,-\nabla f^\alpha\right]$ agree of this space (here we extend the manifold to its closure). Quotienting out the $\mR$ action, we see that the orientation $[\dot u]=[\dot u(\infty)]$ agrees with $\left[\frac{d}{dR}\rho(R)\right]$ which agrees with $\left[ \frac{d}{dR}\Bigr|_{R=R_0} \gamma_u\#^1_Rv\right]$. The orientation map $\sigma^1$ preserves the orientation. 
\begin{figure}
\def\svgwidth{.6\textwidth}\begin{center}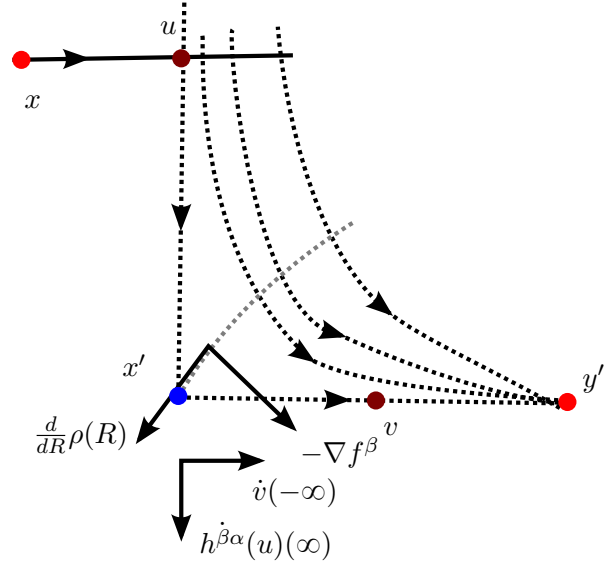\end{center}
\caption{Orientation issues due to breaking in the codomain. The orientations $[\dot{h^{\beta\alpha}}(u)(\infty),\dot v(-\infty)]$ and $\left[ \frac{d}{dR} \rho(R),-\nabla f^\beta\right]$ agree, from which it follows that the map $\sigma^2$ is orientation reversing.}
\label{fig:orientation2}
\end{figure}
The proof that $\sigma^2$ is orientation reversing is analogous. Again we use the notation of Proposition~\ref{prop:gluing2}, with $u\in W_{h^{\beta\alpha}}(x,x')$ and $v\in W(x',y')$. We have the exact sequences of oriented vector spaces
% \begin{equation}
\begin{gather}
\xymatrix@C=.25cm{0\ar[r]&T_vW(x',y')\ar[r]&T_vW^u(x')\ar[r]&N_v W^s(y')\ar[r]&0},
\label{eq:2or1}\\
%\end{equation}
%and
%\begin{equation}
\xymatrix@C=.5cm{0\ar[r]&T_uW_{h^{\beta\alpha}}(x,x')\ar[r]&T_uW^u(x)\ar[r]^-{dh^{\beta\alpha}}&N_{h^{\beta\alpha}(u)} W^s(x')\ar[r]&0},
\label{eq:2or2}\\
%\end{equation}
% and
% \begin{equation}
\xymatrix@C=.23cm{0\ar[r]&T_{u\#^2_{R_0}\gamma_v}W_{h^{\beta\alpha}}(x,y')\ar[r]&T_{u\#^2_{R_0}\gamma_v}W^u(x)\ar[rr]^-{dh^{\beta\alpha}}&&N_{h^{\beta\alpha}({u\#^2_{R_0}\gamma_v)}} W^s(y')\ar[r]&0}.
\label{eq:2or3}
%\end{equation}
\end{gather}
By isomorphisms induced by parallel transport, analogous to the isomorphisms \bref{eq:iso1},\bref{eq:iso2}, and \bref{eq:iso3}, and from \bref{eq:2or2} and \bref{eq:2or1} we get that
\begin{align*}
T_uW^u(x)\cong T_uW_{h^{\beta\alpha}}(x,x')&\oplus N_{h^{\beta\alpha}(u)}(W^s(x')\\&\cong T_uW_{h^{\beta\alpha}}(x,x')\oplus T_vW(x',y')\oplus T_{y'}W^u(y').
\end{align*}
Similarly, from \bref{eq:2or3} and the isomorphisms induced by parallel transport we get
$$
T_uW^u(x)\cong T_{u\#^2\gamma_v}W_{h^{\beta\alpha}}(x,y')\oplus T_{y'}W^u(y'),
$$
which gives that
$$
T_{u\#^2\gamma_v}W_{h^{\beta\alpha}}(x,y') \cong T_uW_{h^{\beta\alpha}}(x,x')\oplus T_{h^{\beta\alpha}(u)}W(x',y').
$$
The orientation on the right hand side is $ab[\dot v]$. Locally around $u$, $h^{\beta\alpha}$ is injective when restricted to $W_{h^{\beta\alpha}}(x,y')\Bigr|_u$. We can therefore take this image and flow with $\psi^\beta$, which we denote by $h^{\beta\alpha}\left(W_{h^{\beta\alpha}}(x,y')\Bigr|_u\right)\times \mR$. The point $x'$ lies in the closure, again the limits of tangent vectors of the orbits through $h^{\beta\alpha}(u)$ and $v$ are well defined for $t\rightarrow \pm \infty$, and the orientations $[h^{\beta\alpha}(u)(\infty),v(-\infty)]$ and $$\left[\frac{d}{dR}\rho(R),-\nabla f^\beta\right]=-\left[-\nabla f^\beta,\frac{d}{dR}\rho(R)\right]$$ 
agree. By quotienting out the flow it follows that $[\dot v]=[\dot v(-\infty)]=-\left[\frac{d}{dR}\rho(R)\right]=-\left[\frac{d}{dR}\Bigr|_{R=R_0}u\#^2_R\gamma_v\right]$. The map $\sigma^2$ is orientation reversing, cf.~Figure~\ref{fig:orientation2}. 
\end{proof}
\subsection{The induced map $h_*^{\beta\alpha}$ is a chain map.}
Propositions~\ref{prop:gluing2} and~\ref{prop:compactness}, along with the considerations of the previous section directly give the following theorem.
\begin{theorem}
Let $h^{\beta\alpha}$ be transverse with respect to $\cQ^\alpha$ and $\cQ^\beta$. For each $x\in \crit f^\alpha$ and $y'\in \crit f^\beta$ with $|x|=|y'|+1$, the space
\begin{eqnarray*}
\widehat{ W_{h^{\beta\alpha}}}(x,y')=W_{h^{\beta\alpha}}(x,y')\bigcup_{|y|=|y'|}&M(x,y)\times W_{h^{\beta\alpha}}(y,y')\\
\bigcup_{|x'|=|x|}& -\left(W_{h^{\beta\alpha}}(x,x')\times M(x',y')\right),
\end{eqnarray*}
has a natural structure as a compact oriented manifold with boundary given by the gluing maps.
\label{prop:compactification}
\end{theorem}
Note that this proposition does not state that exactly one half of the boundary components correspond to $M(x,y)\times W_{h^{\beta\alpha}}(y,y')$ and the other half of the boundary components to $W_{h^{\beta\alpha}}(x,x')\times M(x',y')$. Still it does follow that $h_*^{\beta\alpha}$ is a chain map. 
\begin{proposition}
Let $h^{\beta\alpha}$ be a transverse map with respect to $\cQ^\alpha$ and $\cQ^\beta$. Then the induced map $h_*^{\beta\alpha}:C_*(\cQ^\alpha)\rightarrow C_*(\cQ^\beta)$ is a chain map.
\label{prop:chain}
\end{proposition}
\begin{proof}
Let $x\in \crit f^\alpha$. We compute using Proposition~\ref{prop:compactification}
\begin{align*}
(h_{k-1}^{\beta\alpha}\partial^\alpha_k&-\partial^\beta_k h_k^{\beta\alpha})(x)\\&=\sum_{|y'|=|x|-1}\left(\sum_{|y|=|x|-1}n_{h^{\beta\alpha}}(y,y')n(x,y)-\sum_{|x'|=|x|}n(x',y')n_{h^{\beta\alpha}}(x,x')\right)y'\\
&=\sum_{|y'|=|x|-1}\partial \widehat W_{h^{\beta\alpha}}(x,y')y'=0.
\end{align*}
Because the oriented count of the boundary components of a compact oriented one dimensional manifold is zero. Hence $h_*^{\beta\alpha}$ is a chain map.
\end{proof}

\section{Homotopy induced chain homotopies}
\label{sec:homotopy}
The main technical work, showing that $h_*^{\beta\alpha}$ is a chain map, is done. To show that homotopic maps induce the same maps in Morse homology, we could again define an appropriate moduli space and analyze its compactness failures. However, a simpler method is to construct a dynamical model of the homological cone. This trick is used in Morse homology to show that Morse homology does not depend on the choice of function, metric and orientation. Using the homotopy we build a higher dimensional system -- the dynamical cone -- where we use the fact that an induced map is a chain map to prove homotopy invariance. %, that is Property (iii) from Section~\ref{sec:introfunctorialitymorse}.

\begin{proposition}[Homotopy invariance]
Let $\cQ^\alpha, \cQ^\beta, \cQ^\gamma, \cQ^\delta$, be Morse data with $M^\alpha=M^\gamma$, and $M^\beta=M^\delta$. Let $h^{\beta\alpha}\in \scrT(\cQ^\alpha,\cQ^\beta)$ and $h^{\delta\gamma}\in \scrT(\cQ^\gamma,\cQ^\delta)$, and assume the maps are homotopic. Then $h_*^{\delta\gamma}\Phi^{\gamma\alpha}_*$ and $\Phi^{\delta\beta}_*h_*^{\beta\alpha}$ are chain homotopic. That is, there exists a degree $+1$ map $P_*^{\delta\alpha}:C_*(\cQ^\alpha)\rightarrow C_*(\cQ^\delta)$, such that
\[
\Phi_k^{\delta\beta}h^{\beta\alpha}_k-h^{\delta\gamma}_k\Phi_k^{\gamma\alpha}=-\partial^\delta_{k+1}P_k^{\delta\alpha}-P_{k-1}^{\delta\alpha}\partial_k^{\alpha},\quad\text{for all}\quad k.
\]
\label{prop:homotopy}
\end{proposition}
\begin{proof}
Let $h_\lambda:M^{\alpha}\rightarrow M^{\beta}$ be a smooth homotopy between $h^{\beta\alpha}$ and $h^{\delta\gamma}$. Let $g^{\gamma\alpha}_\lambda$ be a smooth homotopy between $g^\alpha$ and $g^\gamma$, and $g^{\delta\beta}_\lambda$ a smooth homotopy between $g^\beta$ and $g^\delta$. Similarly let $f^{\gamma\alpha}_\lambda$ be a smooth homotopy between $f^{\alpha}$ and $f^\gamma$, and let $f^{\delta\beta}_\lambda$ be a smooth homotopy between $f^{\beta}$ and $f^{\delta}$. Choose $0<\epsilon<\frac{1}{4}$ and let $\omega:\mR\rightarrow [0,1]$ be a smooth, even and $2$-periodic function with the following properties:
$$
\omega(\mu)=\begin{cases}0\qquad&-\epsilon<\lambda<\epsilon\\1&-1\leq\lambda<-1+\epsilon\quad \text{and} \quad 1-\epsilon<\lambda \leq 1.\end{cases}
$$
and $\omega'(\mu)<0$  for $\mu\in (-1+\epsilon,-\epsilon)$, and $\omega'(\mu)>0$ for $\mu\in(\epsilon,1-\epsilon)$. We identify $\mS^1$ with $\mR/2\mZ$. Under the identification of $\mS^1$ with $\mR/2\mZ$, the function $\omega$ descends to a smooth function $\mS^1\rightarrow[0,1]$, which is also denoted $\omega$. Let $r>0$. We define the functions $F^\alpha$ on $M^{\alpha}\times \mS^1$ and $F^\beta$ on $M^{\beta}\times \mS^1$ via
\begin{align*}
F^\alpha(x,\mu)&=f^{\gamma\alpha}_{\omega(\mu)}(x)+r(1+\cos(\pi\mu)),\\
F^\beta(x,\mu)&=f^{\delta\beta}_{\omega(\mu)}(x)+r(1+\cos(\pi\mu)).
\end{align*}
For $r$ sufficiently large the functions $F^\alpha$ and $F^\beta$ are Morse, cf.\ \cite[Lemma 4.5]{rotvandervorst}, and the critical points can be identified by
\begin{align}
\nonumber
C_k(F^\alpha)&\cong C_{k-1}(f^\alpha)\oplus C_{k}(f^\gamma),\quad \text{and}\\
\quad C_k(F^\beta)&\cong C_{k-1}(f^\beta)\oplus C_{k}(f^\delta).
\label{eq:splitting}
\end{align}
We define $H:M^{\alpha}\times\mS^1\rightarrow M^{\beta}\times\mS^1$ with $H(x,\mu)=(h_{\omega(\mu)}(x),\mu)$. The connections of the gradient flow at $\mu=0$ and $\mu=1$ are transverse, and the map $H$ restricted to neighborhood at $\mu=0$ and $\mu=1$ also satisfies the required transversality properties. Hence we can perturb $H$ while keeping it fixed in the neighborhoods of $\mu=0$ and $\mu=1$, as well as the metrics $G^\alpha=g^{\gamma\alpha}_{\omega(\mu)}(x)\oplus d\mu^2$, and $G^\beta=g^{\delta\beta}_{\omega(\mu)}(x)\oplus d\mu^2$ outside $-\epsilon<\mu<\epsilon$ and $1-\epsilon<\mu<1+\epsilon$ to obtain Morse-Smale flows on $M^{\alpha}\times \mS^1$ and $M^{\beta}\times \mS^1$, such that the map $H$ is transverse everywhere. We orient the unstable manifolds in $M^\alpha\times \mS^1$ by $\cO^\alpha=(\partial \mu\oplus \co^\alpha)\cup \co^\gamma$, and the unstable manifolds in $M^\beta\times \mS^1$ by $\cO^\beta=(\partial \mu\oplus \co^\beta)\cup \co^\delta$.

% We study the unstable manifolds of $F^\alpha$. 
Let $(x,0)\in C_k(F^\alpha)$, thus $x\in C_{k-1}(f^\alpha)$. Then $W^u((x,0))\subset M^{\alpha}\times \mS^1\setminus\{1\}$ and $W^u((x,0))\cap M^{\alpha}\times \{0\}=W^u(x)\times\{0\}$. For $(x,1)\in C_k(F^\alpha)$, i.e. $x\in C_k(f^\gamma)$ we have $W^u((x,1))=W^u(x)\times \{1\}$. Similarly for $(y,0)\in C_k(F^\beta)$, i.e. $y\in C_{k-1}(f^\beta)$ we have $W^s((y,0))=W^s(y)\times \{0\}$. Finally for $(y,1)\in C_k(F^\beta)$, i.e. $y\in C_k(f^\delta)$, we have that $W^s((y,1))\subset M^{\beta}\times(\mS^1\setminus\{0\})$, and $W^s((y,1))\cap (M^{\beta}\times\{1\})=W^s(y)\times\{1\}$. 

By Propositions~\ref{prop:compactness} and~\ref{prop:gluing2} and the construction of $H$ it is clear\footnote{The map is isolated~cf.~Definition~\ref{defi:isolatedmap}} that we can restrict the count of the induced map $H_*$ to $W_H(x,y)\cap M^{\alpha}\times [0,1]$, see also Proposition \ref{prop:localmorseinduced}. Note that $H^{-1}(M^{\beta}\times(\mS^1\setminus\{0\}))\subset M^{\alpha}\times (\mS^1\setminus\{0\})$. Let $|(x,0)|=|(y,0)|$, then
$$
W_H((x,0),(y,0))=(W^{u}(x)\cap (h^{\beta\alpha})^{-1}(W^s(y)))\times \{0\},
$$
so $n_H((x,0),(y,0))=n_{h^{\beta\alpha}}(x,y)$ as oriented intersection numbers. % \todo{Note that here $W^u(x)$ and $W^s(y)$ have the same dimension, so there is no minus sign, which does pop up for the boundary map} 
Similarly for $|(x,1)|=|(y,1)|$, we find that
$$
W_H((x,1),(y,1))=(W^u(x)\cap (h^{\delta\gamma})^{-1}(W^s(y)))\times\{1\},
$$
which gives $n_H((x,1),(y,1))=n_{h^{\delta\gamma}}(x,y)$. For $|(x,1)|=|(y,0)|$ we compute that $n_H((x,1),(y,0))=0$. Finally we define a map $P^{\delta\alpha}_k:C_k(f^\alpha)\rightarrow C_{k+1}(f^\delta)$ by counting the intersections of $W^{u}((x,0))$ and $H^{-1}(W^s((y,1))$ with sign. Thus
$$
P^{\delta\alpha}(x)=\sum_{|(x,0)|=|(y,1)|}n_H((x,0),(y,1))y.
$$
With respect to the splittings in Equation~\bref{eq:splitting} the induced map $H_*$ equals
$$
H_k=\left(\begin{matrix}h^{\beta\alpha}_{k-1}&0\\P^{\delta\alpha}_{k-1}&h^{\delta\gamma}_k\\ \end{matrix}\right).
$$
By \cite[Equation (2.11)]{rotvandervorst} the boundary maps $\Delta_k^\alpha$ on $M^{\alpha}\times [0,1]$ and $\Delta^\beta_k\times [0,1]$ on $M^{\beta}$ have the following form
$$
\Delta^\alpha_k=\left(\begin{matrix}-\partial^{\alpha}_{k-1}&0\\\Phi^{\gamma\alpha}_{k-1}&\partial^{\gamma}_k\\ \end{matrix}\right),\qquad \Delta^\beta_k=\left(\begin{matrix}-\partial^{\beta}_{k-1}&0\\\Phi^{\delta\beta}_{k-1}&\partial^{\delta}_k\\ \end{matrix}\right).
$$
Where the $\Phi$'s are the maps that induce isomorphisms in Morse homology. We know that $H_k$ is a chain map, i.e. $H_{k-1}\Delta_{k}^\alpha=\Delta^\delta_{k}H_k$. This implies
$$
\left(\begin{matrix}-h_{k-2}^{\beta\alpha}\partial_{k-1}^\alpha&0\\ -P_{k-2}^{\delta\alpha}\partial^\alpha_{k-1}+h_{k-1}^{\delta\gamma}\Phi^{\gamma\alpha}_{k-1}&h^{\delta\gamma}_{k-1}\partial^\gamma_k\\ \end{matrix}\right)=\left(\begin{matrix}-\partial_{k-1}^\beta h_{k-1}^{\beta\alpha}&0\\ \Phi_{k-1}^{\delta\beta}h^{\beta\alpha}_{k-1}+\partial^{\delta}_{k}P^{\delta\alpha}_{k-1}&\partial^\delta_kh^{\delta\gamma}_{k}\\ \end{matrix}\right).
$$
The lower left corner of this matrix equation gives the desired identity. 
\end{proof}

\section{Composition induced chain homotopies}
\label{sec:functoriality}
To show that compositions of map induce the same map as the compositions of the induced maps in Morse homology, we take in spirit also a homotopy between $h^{\gamma\beta}$ and $h^{\beta\alpha}$ and $h^{\gamma\beta}\circ h^{\beta\alpha}$. This is not directly possible as the maps have different domains and codomains, but we have an approximating homotopy, which is sufficient.   
% \begin{lemma}
% Let $\cQ^\alpha,\cQ^\beta$ and $\cQ^\gamma$ be Morse-Smale triples. Denote by $\psi_R^{\beta}\circ h^{\beta\alpha}(x)=\psi^{\beta}(R,h^{\beta\alpha}(x))$. Assume that $h^{\gamma\beta}\circ\psi_R^{\beta}\circ h^{\beta\alpha}$ is transverse for $R>0$ sufficiently large. Then there is an $R_0>0$ such that for all $R\geq R_0$ we have that 
% $$
% \left(h^{\gamma\beta}\circ\psi_R^{\beta}\circ h^{\beta\alpha}\right)_*=\left(h^{\gamma\beta}\circ\psi_{R_0}^{\beta}\circ h^{\beta\alpha}\right)_*.
% $$
% \label{lem:nobreaking}
% \end{lemma}
%  \begin{proof}
% \todo{How to prove this? The idea is that the homotopy is constant for $R$ sufficiently large....}\todo{I also have another proof, but I like the homotopy here...}
% \end{proof}

% Assume the situation of the lemma.

For a flow $\phi:\mR\times M\rightarrow M$ on a manifold $M$ we write $\phi_R:M\rightarrow M$ for the time $R$ map $\phi_R(x)=\phi(R,x)$. Let $h^{\beta\alpha}\in \scrT(\cQ^\alpha,\cQ^\beta)$ and $h^{\gamma\beta}\in \scrT(\cQ^\beta,\cQ^\gamma)$. We assume, up to possibly a perturbation of $h^{\gamma\beta}$, that the map $H:M^\alpha\times (0,\infty)\rightarrow M^\beta$ defined by
$$
H(p,R)=h^{\gamma\beta}\circ\psi_R^\beta \circ h^{\beta\alpha}(p)
$$
is transverse in the sense that
$$
H\bigr|_{W^u(x)\times (0,\infty)}\pitchfork W^s(z),
$$
for all $x\in \crit f^\alpha$ and $z\in \crit f^\gamma$. We then have moduli spaces
% For $x\in \crit f^\alpha$ and $y\in \crit f^\gamma$, with $|x|=|z|$ consider the moduli space
\begin{align*}
W_{h^{\gamma\beta},h^{\beta\alpha}}(x,z)=\{(p,R)\in &M^\alpha\times (0,\infty)\,|\,p\in W^u(x),\quad\text{and}\\ & h^{\gamma\beta}\circ\psi_{R}\circ h^{\beta\alpha}(p)\in W^s(z)\},
\end{align*}
of dimension $|x|-|z|+1$. The compactness issues if $R\rightarrow \infty$ are due to the breaking of orbits in $M^\beta$, which is described in the following proposition.
\begin{proposition}[Compactness]
Assume the situation as above, with $|x|=|z|$. Let $(p_k,R_k)\in W_{h^{\gamma\beta},h^{\beta\alpha}}(x,z)$ be a sequence with $R_k\rightarrow \infty$ as $k\rightarrow \infty$. Then there exists $y\in \crit f^\beta$ with $|x|=|y|=|z|$, and a subsequence $(p_k,R_k)$ such that $p_k\rightarrow p\in W_{h^{\beta\alpha}}(x,y)$, and $\psi_{R_k}^\beta\circ h^{\beta\alpha}(p_k)\rightarrow q\in W_{h^{\gamma\beta}}(y,z)$. 
\label{prop:functcompactness}
\end{proposition}
\begin{proof}
By compactness of $M^\alpha$ we can choose a subsequence of $(p_k,R_k)$ such that $p_k\rightarrow p$, and choose a subsubsequence, such that also $q_k=\psi_{R_k}^\beta\circ h^{\beta\alpha}(p_k)\rightarrow q$. By similar arguments as in Proposition~\ref{prop:compactness}, using the transversality, $p\in W_{h^{\beta\alpha}}(x',y)$, with $|x|\geq |x'|\geq |y|$ and $q\in W_{h^{\gamma\beta}}(y',z')$ with $|y'|\geq |z'|\geq |z|$ where the equality holds if and only if $x=x'$, $z'=z$. Moreover since $q_k$ and $h^{\beta\alpha}(p_k)$ are on the same orbit, we must have that $|y'|\geq |y|$ with equality if and only if $y'=y$. Since $|x|=|z|$, it follows that $x'=x$ and $z'=z$, and therefore also $y'=y$.
\end{proof}
\begin{figure}
\def\svgwidth{.6\textwidth}\begin{center}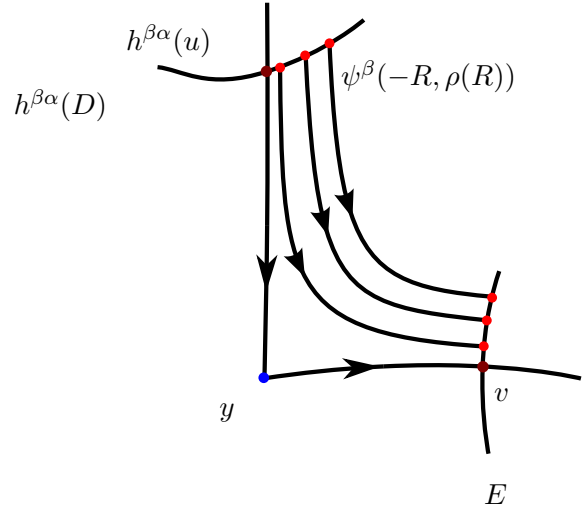\end{center}
\caption{The moduli space $W_{h^{\gamma\beta},h^{\beta\alpha}}(x,z)$ has non-compact ends if $R\rightarrow \infty$. These ends can be compactified by gluing in $W_{h^{\beta\alpha}}(x,y)\times W_{h^{\gamma\beta}}(y,z)$, for all $y\in \crit f^\beta$, with $|x|=|y|=|z|$ as in Proposition~\ref{prop:functgluing}.}
\label{fig:functorialgluing}
\end{figure}

\begin{proposition}[Gluing]
Assume the situation above. Let $x\in \crit f^\alpha$, and $y\in \crit f^\beta$ and $z\in \crit f^\gamma$, with $|x|=|y|=|z|$. Then there exists an $R_0>0$, and a gluing embedding
$$
\#^3: W_{h^{\beta\alpha}}(x,y)\times [R_0,\infty)\times W_{h^{\gamma\beta}}(y,z)\rightarrow W_{h^{\gamma\beta},h^{\beta\alpha}}(x,z).
$$
Moreover, if $(p_k,R_k)\in W_{h^{\gamma\beta},h^{\beta\alpha}}(x,z)$ with $p_k\rightarrow p\in W_{h^{\beta\alpha}}(x,y)$ and $\psi^\beta(R_k,h^{\beta\alpha}(p_k))\rightarrow q\in W_{h^{\gamma\beta}}(y,z)$, then the sequence lies in the image of the embedding for $k$ sufficiently large. 
\label{prop:functgluing}
\end{proposition}
\begin{proof}
See Figure~\ref{fig:functorialgluing}. Let $u\in W_{h^{\beta\alpha}}(x,y)$, and $v\in W_{h^{\gamma\beta}}(y,z)$. By transversality of $h^{\beta\alpha}$, we choose a disc $D^{|y|}\subset W^u(x)$ such that $h^{\beta\alpha}\Bigr|_{D^{|y|}}$ is injective and the image $h^{\beta\alpha}(D^{|y|})$ intersects $W^s(y)$ transversely in $h^{\beta\alpha}(u)$. We also choose a disc $E^{m^\beta-|y|}\subset M^\beta$, intersecting $W^u(y)$ transversely in $v$, whose image is contained in $W^s(z)$, cf.~Figure~\ref{fig:functorialgluing}. By Lemma~\ref{lem:gluing}, we get an $R'_0>0$ and a map $\rho:[R'_0,\infty)\rightarrow M^\beta$. Set $R_0=R'_0/2$ and define $\#^3(u,R,v)=(\left(h^{\beta\alpha}\right)^{-1}(\psi(-R,\rho(R))),2R)$. Here we use the fact that $h^{\beta\alpha}\Bigr|_{D^{|y|}}$ is bijective to $h^{\beta\alpha}(D^{|y|})$. The properties stated follow from Lemma \ref{lem:gluing}.
\end{proof}

\begin{proposition}
Consider the situation above. Then there exist an $R>0$ such that $h^{\gamma\beta}\circ \psi_R^\beta\circ h^{\beta\alpha}\in \scrT(\cQ^\alpha,\cQ^\gamma)$ and the moduli space
$$
W_{h^{\gamma\beta},h^{\beta\alpha}}(x,z,R):=W_{h^{\gamma\beta},h^{\beta\alpha}}(x,z)\bigcap M^\alpha\times (R,\infty),
$$
has a compactification as a smooth oriented manifold with boundary
\begin{align*}
\widehat{W}_{h^{\gamma\beta},h^{\beta\alpha}}(x,z,R)&=W_{h^{\gamma\beta},h^{\beta\alpha}}(x,z,R)\bigcup -W_{h^{\gamma\beta}\circ\psi_R^\beta \circ h^{\beta\alpha}}(x,z)\\&\bigcup_{|y|=|x|} W_{h^{\beta\alpha}}(x,y)\times W_{h^{\gamma\beta}}(y,z).
\end{align*}
\label{prop:functcompactification}
\end{proposition}
\begin{proof}
Proposition~\ref{prop:functcompactness} shows that all sequences $(p_k,R_k)$ in the moduli space with $R_k\rightarrow \infty$ have that the limits $p_k\rightarrow p$ and $h^{\gamma\beta}\circ\psi_{R_k}^\beta\circ h^{\beta\alpha}(p_k)\rightarrow q$ converge in the interior of $W^u(x)$ and $W^s(z)$. By compactness of the domain this implies that there is an $R>0$ such that $W_{h^{\gamma\beta},h^{\beta\alpha}}(x,z,R)$ has no limit points outside $W^u(x)$, and similarly that $h^{\gamma\beta}\circ\psi_R^\beta\circ h^{\beta\alpha}(p_k)$ cannot converge to a point outside $W^s(z)$. By parametric transversality~\cite[Theorem 2.7 page 79]{Hirsch} we can assume that $h^{\gamma\beta}\circ \psi_R^\beta\circ h^{\beta\alpha}\in \scrT(\cQ^\alpha,\cQ^\gamma)$. The space only has two non-compact ends. One is counted by $\bigcup_{|y|=|x|} W_{h^{\beta\alpha}}(x,y)\times W_{h^{\gamma\beta}}(y,z)$ for which we have constructed a gluing map. The other non-compact end is counted $W_{h^{\gamma\beta}\circ\psi_R^\beta \circ h^{\beta\alpha}}(x,z)$ where the gluing map is given by sending $p\in W_{h^{\gamma\beta}\circ \psi_R^\beta\circ h^{\alpha\beta}}(x,z)$ and $R'\in (1,\infty)$ to $(p,R+\frac{1}{R'})\in W_{h^{\gamma\beta},h^{\beta\alpha}}(x,z,R)$.
\end{proof}

\begin{proposition}
Suppose $h^{\beta\alpha}\in \scrT(\cQ^\alpha,\cQ^\beta)$, $h^{\gamma\beta}\in \scrT(\cQ^\beta,\cQ^\gamma)$ and $h^{\gamma\beta}\circ h^{\beta\alpha}\in \scrT(\cQ^\alpha,\cQ^\gamma)$. Then $h_*^{\gamma\beta}\circ h_*^{\beta\alpha}$ and $(h^{\gamma\beta}\circ h^{\beta\alpha})_*$ are chain homotopic, i.e.~there is a degree $+1$ map $P_*^{\gamma\alpha}:C_*(\cQ^\alpha)\rightarrow C_*(\cQ^\gamma)$, such that
$$
h_k^{\gamma\beta} h_k^{\beta\alpha}-\left(h^{\gamma\beta}\circ h^{\beta\alpha}\right)_k=P_{k-1}^{\gamma\alpha}\partial^\alpha_k+\partial_{k+1}^\gamma P_k^{\gamma\alpha},\quad\text{for all}\quad k.
$$
\label{prop:functoriality}
\end{proposition}
\begin{proof}
By Proposition~\ref{prop:functcompactification} we have that 

\begin{align*}
&\left(\left(h^{\gamma\beta}\circ\psi^\beta_R\circ h^{\beta\alpha}\right)_*-h^{\gamma\beta}_*h^{\beta\alpha}_*\right)(x)\\
&\qquad=\sum_{|z|=|x|}\left(n_{h^{\gamma\beta}\circ\psi_R^\beta\circ h^{\beta\alpha}}(x,z)-\sum_{|y|=|x|}n_{h^{\gamma\beta}}(y,z)n_{h^{\beta\alpha}}(x,y)\right)z\\
&\qquad=\sum_{|z|=|x|}\partial\widehat W(x,z,R)z=0.
\end{align*}
The homotopy between $h^{\gamma\beta}\circ h^{\beta\alpha}$ and $h^{\gamma\beta}\circ\psi^\beta_R\circ h^{\beta\alpha}$ sending $\lambda\in[0,1]$ to $h^{\gamma\beta}\circ\psi^\beta_{\lambda R}\circ h^{\beta\alpha}$, induces a chain homotopy by Proposition~\ref{prop:homotopy}. That is, there a degree $+1$ map $P_*$ such that
$$
\left(h^{\gamma\beta}\circ\psi^\beta_R\circ h^{\beta\alpha}\right)_k-\left(h^{\gamma\beta}\circ h^{\beta\alpha}\right)_k=-\partial^\gamma_{k+1}P_k^{\gamma\alpha}-P_{k-1}^{\gamma\alpha}\partial_k^{\alpha}\quad\text{for all} \quad k.
$$
Combining the last two equations gives the chain homotopy.

\end{proof}

\section{Isolation properties of maps}
\label{sec:isolation}

For the remainder of this paper, we do not assume that the base manifolds are necessarily closed. We localize the discussion on functoriality in Morse homology on closed manifolds, and study functoriality for local Morse homology, as well as functoriality for Morse-Conley-Floer homology. For this we need isolation properties of maps.

For a manifold equipped with a flow $\phi$, denote the forwards and backwards orbit as follows
$$
\cO_+(p):=\{\phi(t,p)\,|\, t\geq 0\},\quad 
\cO_-(p):=\{\phi(t,p)\,|\, t\leq 0\}.
$$ 
\begin{definition}
\label{defi:isolatedmap}
Let $M^\alpha,M^\beta$ be manifolds, equipped with flows $\phi^\alpha,\phi^\beta$. Let $N^\alpha,N^\beta$ be isolating neighborhoods. A map $h^{\beta\alpha}:M^\alpha\rightarrow M^\beta$ is an \emph{isolated map} (with respect to $N^\alpha, N^\beta)$ if the set
\begin{equation}
S_{h^{\beta\alpha}}:=\{p\in N^\alpha\,|\, \cO_-(p)\subset N^\alpha, \cO_+(h^{\beta\alpha}(p))\subset N^\beta\},
\label{eq:isolatedmapset}
\end{equation}
satisfies the property that for all $p\in S_{h^{\beta\alpha}}$ we have that
$$
\cO_-(p)\subset \Int N^\alpha, \qquad \cO_+(h^{\beta\alpha}(p))\subset \Int N^\beta.
$$
\end{definition}
Compositions of isolated maps need not be isolated, however the condition is open in the compact-open topology. 
\begin{proposition}
The condition of isolated maps is open: Let $M^\alpha,M^\beta$ be manifolds, equipped with flows $\phi^\alpha,\phi^\beta$. Let $N^\alpha,N^\beta$ be isolating neighborhoods and $h^{\beta\alpha}:M^\alpha\rightarrow M^\beta$ an isolated map. Then there exist an open neighborhood $A$ of $(\phi^\alpha,h^{\beta\alpha},\phi^\beta)$ in
\begin{align*}
\scrC=\{(\tilde\phi^\alpha,\tilde h^{\beta\alpha},\tilde \phi^\beta)&\in C^\infty(\mR\times M^\alpha,M^\alpha)\times C^\infty(M^\alpha, M^\beta)\times C^\infty(\mR\times M^\beta,M^\beta)\,|\, \\
&N^\alpha, N^\beta \text{ are isolating neighborhoods of flows } \tilde\phi^\alpha,\tilde \phi^\beta\}
\end{align*}
equipped with the compact open topology\footnote{The space $\scrC$ itself is open in the space of triples $(\tilde\phi^\alpha,\tilde h^{\beta\alpha},\tilde \phi^\beta)$ with $\tilde \phi^\alpha$ and $\tilde\phi^\beta$ flows, cf.\ \cite[Proposition 3.8]{rotvandervorst}.} such $\tilde h^{\beta\alpha}$ is isolated with respect to $N^\alpha,N^\beta$ and flows $\tilde \phi^\alpha$ and $\tilde \phi^\beta$. 
\label{prop:isolatedmapsareopen}
\end{proposition}
\begin{proof}
Define the sets
\begin{align*}
S^\alpha_-&:=\{p\in N^\alpha\,|\, \cO_-(p)\subset N^\alpha\}=\bigcap_{T\leq 0} \phi^\alpha([T,0],N^\alpha),\\
S^\alpha_+&:=\{p\in N^\alpha\,|\, \cO_+(p)\subset N^\alpha\}=\bigcap_{T\geq 0} \phi^\alpha([0,T],N^\alpha),
\end{align*} 
which are compact. Then $S_{h^{\beta\alpha}}=S^\alpha_-\bigcap (h^{\beta\alpha})^{-1}(S^\beta_+)$. Note that a map being isolated is equivalent to the following two properties of points on the boundary of the isolating neighborhoods in the domain and the codomain.
\begin{enumerate}
\item[(D)] For all $p\in \partial N^\alpha$ either
\begin{enumerate}
\item[(ai)] There exists a $t<0$ such that $\phi^\alpha(t,p)\in M^\alpha\setminus N^\alpha$.
\item[(aii)] $\cO_-(p)\subset N^\alpha$. Then 
$$
T=\sup\{t\in \mR_{\geq 0}\,|\,\phi^\alpha([0,t],p)\subset N^\alpha\}
$$
is finite\footnote{Note that it cannot be the case that $\cO_+(p)$ is also contained in $N^\alpha$ since $N^\alpha$ is an isolating neighborhood of the flow $\phi^\alpha$ and the orbit through a boundary point must leave the isolating neighborhood at some time.}, and for all $q\in \phi^{\alpha}([0,T],p)$ there exists $s\geq 0$ such that $\phi^\beta(s,h^{\beta\alpha}(q))\in M^\beta\setminus N^\beta$.
\end{enumerate}
\item[(CD)] For all $p\in \partial N^\beta$ either
\begin{enumerate}
\item[(bi)] There exists a $t>0$ such that $\phi^\beta(t,p)\in M^\beta\setminus N^\beta.$
\item[(bii)] $\cO_+(p)\subset N^\beta$, and 
$$
T=\inf\{t\in \mR_{\leq 0}\,|\, \phi^{\beta}([t,0],p)\subset N^\beta\}
$$
is finite and $h^{\beta\alpha}(S_-^\alpha)\cap \phi^\beta([T,0],p)=\emptyset$.
\end{enumerate}
\end{enumerate}

In each of these cases we construct open sets $U^{\alpha/\beta}_p\subset M^{\alpha/\beta}$ and $A^{\alpha/\beta}_p\subset \scrC$ such that, for all $q\in U_p$ and $(\tilde \phi^\alpha,\tilde h^{\beta\alpha},\tilde \phi^\beta)\in A^{\alpha/\beta}_p$, these properties remain true for points on the boundary of the isolating neighborhoods. By compactness of $N^\alpha$ and $N^\beta$ we can choose finite number of $U^{\alpha/\beta}_p$ that still cover $\partial N^{\alpha/\beta}$. Then 
$$A=\left(\bigcap_jA^\alpha_{p_j}\right)\bigcap \left(\bigcap_jA^\beta_{p_j}\right)$$ 
is the required open set. 

Let $p\in \partial N^\alpha$ and assume that Property (ai) holds. Then there exists a $t<0$ such that $\phi^\alpha(t,p)\in M^\alpha\setminus N^\alpha$. By continuity there exists an open neighborhood $U_p\ni p$ such that $\phi^\alpha(t,\overline {U_p})\subset M^\alpha\setminus N^\alpha$. Define
$$
A^\alpha_p=\{(\tilde\phi^\alpha,\tilde h^{\beta\alpha},\tilde \phi^\beta)\in \scrC\,|\, \tilde \phi^\alpha(t,\overline{U_p})\subset M^\alpha\setminus N^\alpha\}. 
$$
This is by definition open in the compact-open topology. Then for all $q\in U_p$ and all $(\tilde\phi^\alpha,\tilde h^{\beta\alpha},\tilde \phi^\beta)\in A_p^\alpha$ there exists a $t<0$ such that $\tilde \phi^{\alpha}(t,q)\in M^\alpha\setminus N^\alpha$.

Now let $p\in \partial N^\alpha$ and assume that Property (aii) holds. Choose $t\in [0,T]$ and $s>0$ such that $\phi^\beta(s,h^{\beta\alpha}(\phi^\alpha(t,p)))\in M^\beta\setminus N^\beta$. By continuity of $\phi^{\beta}$ there exists a neighborhood $V^\beta_{t,p}\ni h^{\beta\alpha}(\phi^\alpha(t,p))$ such that $\phi^\beta(s,\overline{V^\beta_{t,p}})\subset M^\beta\setminus N^\beta$. Again by continuity there exists a neighborhood $V_{t,p}^\alpha\ni\phi^\alpha(t,p)$ such that $h^{\beta\alpha}(\overline{V^\alpha_{t,p}})\subset V^\beta_{t,p}$. Since $\phi^\alpha([0,T],p)$ is compact and covered by the $V^\alpha_{t,p}$, we can choose a finite number of $t_j$ such that the sets $V_{t_j,p}^\alpha$ cover $\phi^{\alpha}([0,T],p)$. Now there exists an $\epsilon>0$ small and a neighborhood $V_p\ni p$ such that $\phi^{\alpha}(T+\epsilon,\overline{V_p^\alpha})\in (M^\alpha\setminus N^\alpha)\bigcap (\bigcup_j V^\alpha_{t_j,p})$ and $\phi^\alpha([0,T+\epsilon],\overline{V_p^\alpha})\subset \bigcup_j V^\alpha_{t_j,p}$. Set
\begin{align*}
U_p^\alpha&=\left(\bigcap_j \phi^{\alpha}(-t_j,V^\alpha_{t_j,p})\right)\bigcap V_p,\\
A_p^\alpha&=\{(\tilde\phi^\alpha,\tilde h^{\beta\alpha},\tilde \phi^\beta)\in \scrC\,|\, \tilde \phi^\alpha(t+\epsilon,\overline{ V_p^\alpha})\subset (M^\alpha\setminus N^\alpha)\bigcap \left(\bigcup_j V^\alpha_{t_j,p}\right),\\
&\quad\tilde\phi^\alpha([0,T+\epsilon],\overline{V_p^\alpha})\subset \bigcup_j V^\alpha_{t_j,p},\quad \tilde h^{\beta\alpha}(\overline{V_{t_j,p}^\alpha})\subset V_{t_j,p}^\beta, \quad\tilde \phi^{\beta}(s_j,\overline{ V_{t_j,p}^\beta})\subset M^\beta\setminus N^\beta\}.
\end{align*}
Note that, for all $q\in U_p^\alpha$ and $(\tilde\phi^\alpha,\tilde h^{\beta\alpha},\tilde \phi^\beta)\in A_p^\alpha$ and all $t\geq 0$ such that $\tilde \phi^\alpha([0,t],q)\in N^\alpha$, there exists an $s\geq 0$ such that $\tilde \phi^\beta(s,\tilde h^{\beta\alpha}(t,q))\in M^\beta\setminus N^\beta$. It is not necessarily true that $\cO_-(q)\subset N^\alpha$, but this does not matter. Now, by compactness there exists a finite number of $p_i\in \partial N^\alpha$ such that the corresponding $U_{p_i}$ cover $\partial N^\alpha$. Then $A^\alpha=\bigcap_i A^\alpha_{p_i}$ is open, and for all $p\in \partial N^\alpha$ either Property (ai) or (aii) holds with respect to each $(\tilde\phi^\alpha,\tilde h^{\beta\alpha},\tilde \phi^\beta)\in A^\alpha$.

Now we study the codomain. Let $p\in \partial N^\beta$ and assume that Property (bi) holds. Then completely analogously to the situtation of Property (ai) there exists a $t>0$ and an open $U^\beta_p$ such that $\phi^{\beta}(t,\overline{U_p^\beta})\subset M^\beta\setminus N^\beta$. Define
$$
A_p^\beta=\{(\tilde\phi^\alpha,\tilde h^{\beta\alpha},\tilde \phi^\beta)\in \scrC\,|\, \tilde\phi^{\beta}(t,\overline{U_p^\beta})\subset M^\beta\setminus N^\beta\}.
$$
Then for all $q\in U^\beta_p$ and all $(\tilde\phi^\alpha,\tilde h^{\beta\alpha},\tilde \phi^\beta)\in A^\beta_p$ there exists a $t>0$ such that $\tilde \phi^\beta(t,q)\in N^\beta$. 

Finally let $p\in \partial N^\beta$ and assume that Property (bii) holds. Thus $\cO_+(p)\subset N^\beta$ and $T=\inf\{t\in \mR_{\leq 0}\,|\,\phi^\beta([T,0],p)\subset N^\beta\}>-\infty$ and $h^{\beta\alpha}(S_-)\cap \phi^{\beta}([T,0],p)=\emptyset$. Since both sets are compact, there exists an open $V\supset \phi^\beta([T,0],p)$ such that $h^{\beta\alpha}(S_-)\cap \overline{V}=\emptyset$. This implies that for all $q\in \bigl(h^{\beta\alpha}\bigr)^{-1}(\overline V)\cap N^\alpha$ there exists an $s<0$ such that $\phi^{\alpha}(s,q)\in M^\alpha\setminus N^\alpha$. But then there exists an open $V^\alpha_q\ni q$ such that $\phi^\alpha(s,\overline{V^\alpha_q})\subset M^\alpha\setminus N^\alpha$. The set $\bigl(h^{\beta\alpha}\bigr)^{-1}(\overline V)\cap N^\alpha$ is compact hence it is covered by $V^\alpha_{q_i}$ for a finite number of $q_i$. Note that, for all flows $\tilde \phi^\alpha$ with $\tilde \phi^\alpha(s_i,\overline{V^\alpha_{q_i}})\subset M^\alpha\setminus N^\alpha$, we have that $\tilde S_-^\alpha=\{p\in N^\alpha\,|\,\tilde\phi^\alpha(t,p)\in N^\alpha,\,\forall t\leq 0\}$ is contained in $N^\alpha \setminus \bigcup_i\overline V_{q_i}^\alpha$. Let us return to the codomain. By continuity of the flow and the choice of $T$ there exists an $\epsilon>0$ such that $\phi^\beta(T-\epsilon,p)\subset (M^\alpha \setminus N^\alpha)\bigcap V$ with $\phi^\beta([T-\epsilon,0],p)\subset V$. Fix a neighborhood $V_p^\beta\ni p$ such that $\phi^\beta(T-\epsilon,\overline{V_p^\beta})\subset (M^\beta\setminus N^\beta)\bigcap V)$. Now for $t\in [T,0]$ the sets $\phi^\beta(t,V_p^\beta)\cap V$ cover $\phi^\beta([T,0],p)$. Using compactness, choose a finite number of times $t_j$ such that $\phi^\beta(t_j,V_p^\beta)\cap V$. For all $q\in \bigcap_j\phi(-t_j,\phi^\beta(t_j,V_p^\beta)\cap V)$ we have that $T_q=\inf\{t\in \mR_{\leq 0}\,|\,\phi^{\beta}([t,0],q)\}\geq T-\epsilon$, and for all $t\in T_q$, we have $\phi^\beta(t,q)\subset V$. Define
\begin{align*}
A^\beta_p&=\{(\tilde \phi^\alpha,\tilde h^{\beta\alpha},\phi^\beta)\,|\, \tilde \phi(s_i,\overline{ V_{q_i}^\alpha})\subset M^\alpha\setminus N^\alpha\\
&\quad \tilde h^{\beta\alpha}(N^\alpha\setminus \bigcup V_{q_i}^\alpha)\subset M^\beta\setminus \overline{V},\quad \tilde \phi^\beta([T-\epsilon]\times \overline {U^\beta_p})\subset V\}.
\end{align*}
Then for all $q\in U_p^\beta$ and all $(\tilde \phi^\alpha,\tilde h^{\beta\alpha},\tilde\phi^\beta)\in A_p^\beta$
$$
\tilde h^{\beta\alpha}(\tilde S^\alpha_-)\cap \tilde \phi^{\beta}([T_q,0],q)=\emptyset.
$$
The sets $U^\beta_p$ cover $\partial N^\beta$. Hence we can choose a finite number of $p_j$ such that for all $p\in \partial N^\beta$ and all $(\tilde \phi^\alpha,\tilde h^{\beta\alpha},\phi^\beta)\in A^\beta=\bigcap_j A^\beta_{p_j}$ either Property (bi) or (bii) holds. 

Set $A=A^\alpha\cap A^\beta$. Then, by construction of the set $A$, for all $(\tilde \phi^\alpha,\tilde h^{\beta\alpha},\tilde\phi^\beta)\in A$ and all $p\in N^\alpha$ Properties (ai) and (aii) hold. And similarly for all $p\in \partial N^\beta$, Properties (bi) and (bii) hold. Thus $\tilde h^{\beta\alpha}$ is isolated with respect to the flows $\tilde \phi^\alpha,\tilde \phi^\beta$, and isolating neighborhoods $N^\alpha,N^\beta$.
\end{proof}

\begin{definition}
A homotopy $h_\lambda$ is \emph{isolated} with respect to isolated homotopies of flows $\phi_\lambda^\alpha$ and $\phi^{\beta}_\lambda$ if each $h_\lambda$ is an isolated map with respect to $\phi^\alpha_\lambda$ and $\phi^\beta_\lambda$. The maps $h_0$ and $h_1$ are said to be isolated homotopic to each other.
\end{definition}

We remark that isolated homotopies are also open in the compact open topology. An isolated homotopy $h_\lambda$ can be viewed as an isolated map $H:M^\alpha\times[0,1] \rightarrow M^\beta\times[0,1]$, with $H(x,\lambda)=(h_\lambda(x),\lambda)$ with flows $\Phi^\alpha=(\phi_\lambda^\alpha,\lambda)$. Thus the set of isolated homotopies is open in the compact open topology by Proposition~\ref{prop:isolatedmapsareopen}. 
Flow maps, cf.~Definition~\ref{defi:flowmap}, are isolated with respect to well chosen isolating neighborhoods. 

\begin{proposition}
Let $h^{\beta\alpha}:M^\alpha\rightarrow M^\beta$ be a flow map. Then for each isolating neighborhood $N^\beta$ of $\phi^\beta$, $N^\alpha=(h^{\beta\alpha})^{-1}(N^\beta)$ is an isolating neighborhood, and $S^{\alpha}=\Inv (N^{\alpha},\phi^\alpha)=(h^{\beta\alpha})^{-1}(\Inv( N^\beta,\phi^\beta))=(h^{\beta\alpha})^{-1}(S^\beta)$. Moreover $h^{\beta\alpha}$ is isolated w.r.t.~$N^\alpha$ and $N^\beta$. If $h^{\gamma\beta}$ is another flow map, and $N^\gamma$ is an isolating neighborhood then $h^{\gamma\beta}\circ \phi^\beta_R\circ h^{\beta\alpha}$ is isolated for all $R$ with respect to  $N^\alpha=(h^{\gamma\beta}\circ h^{\beta\alpha})^{-1}(N^\gamma)$.
\label{prop:flowmaps}
\end{proposition}
\begin{proof}
We follow McCord \cite{mccord}. Since $h^{\beta\alpha}$ is proper $N^\alpha$ is compact. If $p\in \Inv(N^{\alpha})$ then $\phi^{\alpha}(t,p)\in N^\alpha$ for all $t\in \mR$. By equivariance $\phi^{\beta}(t,h^{\beta\alpha}(p))\in N^\beta$ for all $t$ and since $N^\beta$ is an isolating neighborhood $h^{\beta\alpha}(p)\in \Int(N^\beta)$. Thus $p\in (h^{\beta\alpha})^{-1}(\Int N^\beta)\subset \Int N^\alpha$. If $p\in S^\alpha$, then for all $t$ we have that $\phi^\alpha(t,p)\in N^\alpha$, and thus $h^{\beta\alpha}(\phi^\alpha(t,p))\in N^\beta$. By equivariance it follows that $\phi^{\beta}(t,h^{\beta\alpha}(p))\in N^\beta$ for all $t$. Hence $h^{\beta\alpha}(p)\in S^\beta$. Analogously, if $p\in (h^{\beta\alpha}(p))^{-1}(S^\beta)$ then $\phi^\beta(t,h^{\beta\alpha})\in S^\beta$ for all $t$. By equivariance it follows that $h^{\beta\alpha}(\phi^{\alpha}(t,p))\in S^\beta$ and this implies that $\phi(t,p)\in (h^{\beta\alpha})^{-1}(N^\beta)$ for all $t$. We finally show that $h^{\beta\alpha}$ is isolated. If $p\in S_{h^{\beta\alpha}}$, then $\phi^{\alpha}(t,p)\in N^\alpha$ for all $t<0$, and $\phi^{\beta}(t,h^{\beta\alpha}(p))\in N^\beta$ for all $t>0$. By equivariance $h^{\beta\alpha}(\phi^{\alpha}(t,p))\in N^\beta$ for all $t>0$, and thus $\phi^{\alpha}(t,p)\in N^\alpha$ for all $t$. Thus $p\in S^\alpha$ and $\cO-(p)\subset S^\alpha\subset \Int N^\alpha$. Similarly $h^{\beta\alpha}(p)\in S^\beta$ thus $\cO_+(h^{\beta\alpha}(p))\subset S^\beta \subset \Int N^\beta$.  
The proof of the latter statement follows along the same lines. 
\end{proof}

The previous proposition states that we can pull back isolated invariant sets and neighborhoods along flow maps. The same is true for Lyapunov functions.

\begin{proposition}
Let $h^{\beta\alpha}:M^\alpha\rightarrow M^\beta$ be a flow map with respect to $\phi^\alpha$ and $\phi^\beta$. Let $S^\beta\subset M^\beta$ be an isolated invariant set, and $f^\beta:M^\beta\rightarrow \mR$ a Lyapunov function satisfying the Lyapunov property with respect to an isolating neighborhood $N^\beta$. Then $f^\alpha:=f^\beta\circ h^{\beta\alpha}$ is a Lyapunov function for $S^\alpha:=(h^{\beta\alpha})^{-1}(S^\beta)$ satisfying the Lyapunov property on $N^\alpha=(h^{\beta\alpha})^{-1}(N^\beta)$. Moreover, for any metric $e^\alpha,e^\beta$ the map $h^{\beta\alpha}$ is isolated with respect to the gradient flows $\psi^\alpha$ and $\psi^\beta$ of $(f^\alpha,e^\alpha)$ and $(f^\beta,e^\beta)$.  
\label{prop:lyapunovpullback}
\end{proposition}
\begin{proof}
Since $h^{\beta\alpha}(S^\alpha)\subset S^\beta$ and $f^\beta\bigr|_{S^\beta}\equiv c$ is constant it follows that $f^\alpha\bigr|_{S^\alpha}\equiv c$ is constant. If $p\in N^\alpha\setminus S^\alpha$ then there exists a $t\in \mR$ such that $\phi^\alpha(t,p)\subset M^\alpha\setminus N^\alpha$. Then $h^{\beta\alpha}(\phi^{\alpha}(t,p))\in M^\beta\setminus N^\beta$ and by equivariance $\phi^\beta(t,h^{\beta\alpha}(p))\in M^\beta\setminus N^\beta$. Therefore $h^{\beta\alpha}(p)\in N^\beta\setminus S^\beta$. Again by equivariance
$$
\frac{d}{dt}\Bigr|_{t=0}f^\alpha(\phi^\alpha(t,p))=\frac{d}{dt}\Bigr|_{t=0}f^{\beta}(h^{\beta\alpha}(\phi^\alpha(t,p)))=\frac{d}{dt}\Bigr|_{t=0}f^{\beta}(\phi^\beta(t,h^{\beta\alpha}(p)))<0.
$$
Thus $f^\alpha$ is a Lyapunov function. We now prove that $h^{\beta\alpha}$ is an isolated map with respect to the isolating neighborhoods $N^\alpha$ and $N^\beta$ for the gradient flows $\psi^\alpha$ and $\psi^\beta$. The neighborhoods are isolating for the gradient flows by~\cite[Lemma 3.3]{rotvandervorst}. Let $P_{h^{\beta\alpha}}$ be the set of Equation~\bref{eq:isolatedmapset} for the gradient flows. By the arguments of \cite[Lemma 3.3]{rotvandervorst} we have that, for $p\in P_{h^{\beta\alpha}}$ that
$$
\alpha(p)=\{\lim_{n\rightarrow \infty}\psi^\alpha(t_n,p)\,|\,\lim_{n\rightarrow \infty} t_n=-\infty\}\subset \crit f^\alpha\cap N^\alpha,
$$
and
$$
\omega(h^{\beta\alpha}(p))=\{\lim_{n\rightarrow \infty}\psi^\beta(t_n,h^{\beta\alpha}(p))\,|\,\lim_{n\rightarrow \infty} t_n=\infty\}\subset \crit f^\beta\cap N^\beta.
$$
Consider $b_p:\mR\rightarrow \mR$ with
$$
b_p(t)=\begin{cases} f^\alpha(\psi^\alpha(t,p))\quad\text{for}\quad& t\leq 0\\
f^\beta(\psi^\beta(t,h^{\beta\alpha}))&t>0.
\end{cases}
$$
The function $b_p$ is continuous, smooth outside zero and by the Lyapunov property $\frac{d}{dt}b_p(t)\leq 0$. By \cite[Lemma 3.1]{rotvandervorst} we know that $\crit f^\alpha\cap N^\alpha\subset S^\alpha$ and $\crit f^\beta\cap N^\beta\subset S^\beta$. Moreover $f^\alpha\bigr|_{S^\alpha}\equiv f^\alpha\bigr|_{S^\alpha}\equiv c$, hence $\lim_{t\rightarrow -\infty}b_p(t)=\lim_{t\rightarrow \infty}b_p(t)=c$. Hence $b_p$ is constant and it follows that $p\in S^\alpha$. Thus the full orbit through $p$ contained in $\Int N^\alpha$ and the full orbit through $h^{\beta\alpha}(p)$ is contained in $\Int N^\beta$. Thus $h^{\beta\alpha}$ is isolated with respect to the gradient flows.
\end{proof}

\section{Functoriality in local Morse homology}
\label{sec:functorialitylocalmorse}

We are interested in the functorial behavior of local Morse homology. Let us recall the definition of local Morse homology more in depth. Suppose $N^\alpha$ is an isolating neighborhood of the gradient flow of $f^\alpha$ an $g^\alpha$. The local stable and unstable manifolds of critical points inside $N^\alpha$ are defined by
\begin{align*}
W^{u}_\text{ loc}(x;N^\alpha)&:=\{p\in N^\alpha\,|\, \psi^{\alpha}(t,p)\in N^\alpha, \forall t<0,\text{ and } \lim_{t\rightarrow-\infty}\psi^\alpha(t,p)=x \},\\
W^{s}_\text{ loc}(x;N^\alpha)&:=\{p\in N^\alpha\,|\, \psi^{\alpha}(t,p)\in N^\alpha, \forall t>0,\text{ and } \lim_{t\rightarrow \infty}\psi^\alpha(t,p)=x \}.
\end{align*}
We write $W_{N^\alpha}(x,y)=W^{u}_\text{ loc}(x,y;N^\alpha)\cap W^s_\text{ loc}(y; N^\alpha)$, and $M_{N^\alpha}(x,y)=W_{N^\alpha}(x,y)/\mR$. We say that the gradient flow is Morse-Smale on $N^\alpha$, cf.\ \cite[Definition 3.5]{rotvandervorst}, if the critical points of $f^\alpha$ inside $N^\alpha$ are non-degenerate, and for each $p\in W_{N^\alpha}(x,y)$, we have that $T_pW^u(x)+T_p W^s(y)=T_pM^\alpha$. The intersection is said to be transverse and we write $W^u_\text{ loc}(x;N^\alpha)\pitchfork W^s_\text{ loc}(y;N^\alpha)$. Denote by $\cP^\alpha=(M^\alpha,f^\alpha,g^\alpha,N^\alpha,\co^\alpha)$ a choice of manifold (not necessarily closed), a function, a metric and an isolating neighborhood of the gradient flow, such that $(f^\alpha,g^\alpha)$ is Morse-Smale on $N^\alpha$, and a choice of orientations of the local unstable manifolds. For such a local Morse datum $\cP^\alpha$, we define the local Morse complex
$$
C_k(\cP^\alpha):=\crit_k f^\alpha\cap N^\alpha,\quad\partial(\cP^\alpha)(x):=\sum_{|y|=|x|-1}n_{N^\alpha,\text{ loc}}(x,y) y.
$$
Where $n_{N^\alpha,\text{loc}}(x,y)$ denotes the oriented count of points in $M_{N^\alpha}(x,y)$. The differential satisfies $\partial^2(\cP^\alpha)=0$, hence we can define local Morse homology. This not an invariant for $N^\alpha$ but crucially depends on the gradient flow. Compare for example the gradient flows of $f^\alpha(x)=x^2$ or $f^\beta(x)=-x^2$ on $\mR$ with isolating neighborhoods $N=[-1,1]$. The homology is invariant under homotopies $(f_\lambda,g_\lambda)$, as long as the gradient flows preserve isolation. The canonical isomorphisms, induced by continuation are denoted by $\Phi^{\beta\alpha}$. The local Morse homology recovers the Conley index of the gradient flow. 

If $h^{\beta\alpha}$ is isolated with respect to $\cP^\alpha$, $\cP^\beta$, we say it is transverse (with respect to $\cP^\alpha$ and $\cP^\beta$), if for all $p\in W^u_\text{ loc}(x;N^\alpha)\cap (h^{\beta\alpha})^{-1}(W^s_\text{ loc}(y,N^\beta))$, we have 
$$
dh^{\beta\alpha}T_p W^u(x)+T_{h^{\beta\alpha}(p)}W^s(y)=T_{h(p)}M^\beta.
$$
We write $W_{h^{\beta\alpha},\text{ loc}}(x,y)=W^u_\text{ loc}(x;N^\alpha)\cap W^s_\text{ loc}(y;N^\beta)$. The oriented intersection number is denoted by $n_{h^{\beta\alpha},\text{loc}}(x,y)$. 

\begin{figure}
\def\svgwidth{.9\textwidth}\begin{center}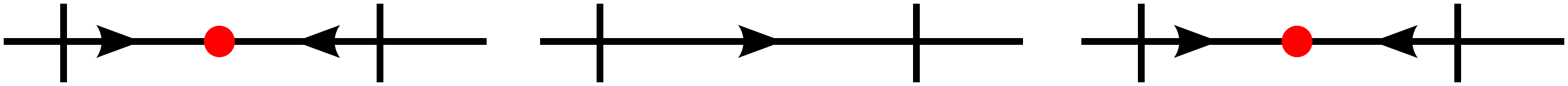\end{center}
\caption{All isolated maps induce chain maps in local Morse homology, but they are not necessarily functorial, as they do not capture the dynamical content. The gradient flows of $f^\alpha(x)=x^2$, $f^\beta=(x-3)^2$, and $f^\gamma(x)=x^2$ on $\mR$, with isolating neighborhood $N=[-1,1]$ are depicted. The identity maps are isolated. We compute that $\id^{\beta\alpha}_*=0$, $\id^{\gamma\beta}_*=0$, but $\id^{\gamma\alpha}_*$ is the identity. The problem is that $\id^{\gamma\beta}\circ\psi^\beta_{R}\circ \id^{\beta\alpha}$ is not isolated for all $R>0$.}
\label{fig:counterexample}
\end{figure}

\begin{proposition}Let $h^{\beta\alpha}\in \scrT(\cP^\alpha,\cP^\beta)$ and suppose $h^{\beta\alpha}$ is isolated with respect to $R^\alpha$ and $R^\beta$. 
\begin{itemize} 
\item Then $h_*^{\beta\alpha}(x):=\sum_{|x|=|y|}n_{h^{\beta\alpha},\text{loc}}(x,y)y$ is a chain map. 
\item Suppose $\cP^\gamma$ and $\cP^\delta$ are different local Morse data with 
$$
M^\alpha=M^\gamma,\quad N^\alpha=N^\gamma,\quad M^\beta=M^\delta\quad\text{and}\quad N^\beta=N^\delta,
$$
such that the gradient flow of $\cP^\alpha$ is isolated homotopic to the gradient flow of $\cP^\gamma$, the gradient flow of $\cP^\beta$ is isolated homotopic to the gradient flow of $\cP^\delta$, and $h^{\beta\alpha}\in \scrT(\cP^\alpha,\cP^\beta)$ and $h^{\delta,\gamma}\in \scrT(\cP^\gamma,\cP^\delta)$ are isolated homotopic through these isolated homotopies. Then $\Phi^{\delta\beta}_*h^{\beta\alpha}_*$ and $h^{\delta\gamma}_*\Phi^{\gamma\alpha}$ are chain homotopic. 
\end{itemize}
% If $h_{\lambda}$ is an isolated homotopy with to isolated homotopies of gradient flows on $M^\alpha$ and $M^\beta$ and $h_0=h^{\beta\alpha}$ is transverse to $\cP^\alpha$ and $\cP^\beta$, and $h_1=h^{\gamma\beta}$ is transverse with respect to $\cP^\gamma=(M^\alpha,f^\gamma,g^\gamma,N^\alpha,\co^\gamma)$ and $\cP^\delta=(M^\beta,f^\delta,g^\delta,N^\beta,\co^\delta)$, then $\Phi^{\delta\beta} h^{\beta\alpha}_*$ and $h^{\delta\gamma}_*\Phi^{\gamma\alpha}_*$ are chain homotopic. 
\label{prop:localmorseinduced}
\end{proposition}
\begin{proof}
We argue that the gluing maps constructed in Proposition~\ref{prop:gluing2} restrict to local gluing maps
\begin{align*}
\#^1:&M_{N^\alpha}(x,y)\times [R_0,\infty)\times W_{h^{\beta\alpha},\text{loc}}(y,y')\rightarrow W_{h^{\beta\alpha},\text{loc}}(x,y')\\
\#^2:&W_{h^{\beta\alpha},\text{loc}}(x,x')\times [R_0,\infty)\times M_{N^\beta}(x',y')\rightarrow W_{h^{\beta\alpha},\text{loc}}(x,y')
\end{align*}
Consider $\#^1$, let $\gamma_u\in M_{N^\alpha}(x,y)$ and $v\in W_{h^{\beta\alpha},\text{loc}}(y,y')$. Geometric convergence implies that the backwards orbit $\cO_-(\gamma_u\#^1_Rv)$ lies arbitrary close, for $R$ sufficiently large, to the images of $\cO(u)$ and $\cO_-(v)$. The latter are contained in $\Int N^\alpha$, hence $\cO_-(\gamma_u\#^1_R v)\subset \Int N^\alpha$. Similarly $\cO_+(h^{\beta\alpha}(u\#^1_Rv)\subset \Int N^\beta$ for $R$ sufficiently large since it must converge to $\cO_+(h^{\beta\alpha}(v))$. Thus for $R_0$ possibly larger than in Proposition~\ref{prop:gluing2} there is a well defined restriction of the gluing map $\#^1$. The situation for $\#^2$ is analogous.

For the compactness issues, observe that $S_{h^{\beta\alpha}}=S_-^\alpha\cap (h^{\beta\alpha})^{-1}(S^\beta_+)$ hence is compact. But it also equals the set $\bigcup_{x,y} W_{h^{\beta\alpha},\text{loc}}(x,y)$. Therefore if $p_k\in W_{h^{\beta\alpha},\text{loc}}(x,y)$ then it has a convergent subsequence $p_k\rightarrow p$ with $p\in W_{h^{\beta\alpha},\text{loc}}(x',y')$. Since $\cO_-(p_k)\subset \Int N^\alpha$ and $\cO_+(h^{\beta\alpha}(p_k))\subset \Int N^\beta$ it follows that the compactness issues are those described in Proposition~\ref{prop:compactness}, but then for the local moduli spaces.

Now consider the isolated homotopy. Isolation implies that we can follow the argument of Proposition~\ref{prop:homotopy} to construct isolating neighborhoods on $M^\alpha\times[0,1]$ and  $M^\beta\times [0,1]$ and a map $H^{\beta\delta}$ that is isolated with respect to these neighborhoods. Then the observation that $H^{\beta\delta}$ is a chain map implies that there exists a chain homotopy. 
\end{proof}

Recall that the local Morse homology of the isolating neighborhood of any pair $(f,g)$, which is not necessarily Morse-Smale, is defined by
$$
HM_*(f,g,N):=\varprojlim HM_*(\cP^\alpha)
$$
where the inverse limit runs over all local Morse data $\cP^\alpha$ whose gradient flows are isolated homotopic to the gradient flows of $(f,g)$ on $N$, with respect to the canonical isomorphisms.
\begin{proposition}
Let $(f^\alpha,g^\alpha)$ and $(f^\beta,g^\beta)$ pairs of functions of metrics on $M^\alpha$ and $M^\beta$, which are not assumed to be Morse-Smale. Suppose that $h^{\beta\alpha}:M^\alpha\rightarrow M^\beta$ is isolated with respect to isolating neighborhoods $N^\alpha$ and $N^\beta$. Then $h^{\beta\alpha}$ induces a map of local Morse homologies
$$
h_*^{\beta\alpha}:HM_*(f^\alpha,g^\alpha,N^\alpha)\rightarrow HM_*(f^\beta,g^\beta,N^\beta).
$$
Suppose $h^{\delta\gamma}$ is isolated homotopic to $h^{\beta\alpha}$ through isolated homotopies of gradient flows between $(f^\alpha,g^\alpha,N^\alpha)$ and $(f^\gamma,g^\gamma,N^\gamma)$ and $(f^\beta,g^\beta,N^\beta)$ and $(f^\delta,g^\delta,N^\delta)$. Then the diagram
$$
\xymatrix{HM_*(f^\alpha,g^\alpha,N^\alpha)\ar[r]^-{h^{\beta\alpha}_*}\ar[d]^{\Phi^{\gamma\alpha}_*}&HM_*(f^\beta,g^\beta,N^\beta)\ar[d]^{\Phi^{\delta\beta}_*}.\\
HM_*(f^\gamma,g^\gamma,N^\gamma)\ar[r]^-{h^{\delta\gamma}_*}&HM_*(f^\delta,g^\delta,N^\delta).}
$$
commutes.
\label{prop:degenerateinduced}
\end{proposition}
\begin{proof}
Let $\cP^\gamma=(M^\alpha,f^\gamma,g^\gamma,N^\alpha,\co^\gamma)$ and $\cP^\delta=(M^\beta,f^\delta,g^\delta,N^\beta,\co^\delta)$ be local Morse data such that 
$$(f^\gamma,g^\gamma)\in \IMS(f^\alpha,g^\alpha,N^\alpha),\quad (f^\delta,g^\delta)\in \IMS(f^\beta,g^\beta,N^\beta),$$ and $h^{\beta\alpha}$ is isolated for the isolated homotopies connecting the $(f^\alpha,g^\alpha)$ with  $(f^\gamma,g^\gamma)$ and $(f^\beta,g^\beta)$ with $(f^\delta,g^\delta)$ and $h^{\beta\alpha}\in \scrT(\cP^\gamma,\cP^\delta)$. This is possible by Proposition~\ref{prop:isolatedmapsareopen}, and \cite[Corollary 3.12]{rotvandervorst}. By the density of transverse maps Theorem~\ref{thm:genericity}, there exist a small perturbation $h^{\delta\gamma}$ isolated homotopic to $h^{\beta\alpha}$ by this homotopy. By Proposition~\ref{prop:localmorseinduced} we get a map
$$
h_*^{\delta\gamma}:HM_*(\cP^\gamma)\rightarrow HM_*(\cP^\delta).
$$
Moreover, given different local Morse data $\cP^\epsilon$ and $\cP^\zeta$ as above, we can construct an isolated map $h^{\zeta\epsilon}$ isolated homotopic to $h^{\beta\alpha}$ and hence also to $h^{\delta\beta}$ by concatenation. From~Proposition~\ref{prop:localmorseinduced} it follows that there is the following commutative diagram
$$
\xymatrix{HM_*(\cP^\gamma)\ar[d]_{\Phi^{\epsilon\gamma}_*}\ar[r]^{h^{\delta\gamma}_*}&HM_*(\cP^\delta)\ar[d]^{\Phi_*^{\zeta\delta}}\\
HM_*(\cP^\epsilon)\ar[r]_{h^{\zeta\epsilon}_*}&HM_*(\cP^\zeta)
}
$$
Which means that we have an induced map $h^{\beta\alpha}_*:HM_*(f^\alpha,g^\alpha,N^\alpha)\rightarrow HM_*(f^\beta,g^\beta,N^\beta)$ between the inverse limits. The arguments for the homotopy of the maps is analogous.
\end{proof}

 The chain maps defined above are not necessarily functorial. Consider for example Figure~\ref{fig:counterexample}. The problem is that, in the proof of functoriality for Morse homology, we need the fact that $h^{\gamma\beta}\circ \psi^\beta_R\circ h^{\beta\alpha}$ is isolated for all $R\geq 0$ to establish functoriality. If we require this almost homotopy to be isolated the proof of functoriality follows mutatis mutandis. 
 \begin{proposition}
Let $h^{\beta\alpha}$ and $h^{\gamma\beta}$ be transverse and isolated. Assume that $h^{\gamma\beta}\circ h^{\beta\alpha}$ is transverse and isolated, and assume that $h^{\gamma\beta}\circ\psi^\beta_R\circ h^{\beta\alpha}$ is an isolated homotopy for $R\in[0,R']$ for each $R'>0$. Then $h^{\gamma\beta}_*h^{\beta\alpha}_*$ and $\left(h^{\gamma\beta}\circ h^{\beta\alpha}\right)_*$ are chain homotopic.
\label{prop:localmorsefunctorial}
\end{proposition}

\begin{proposition}
Let $(f^\alpha,g^\alpha)$, $(f^\beta,g^\beta)$, and $(f^\gamma,g^\gamma)$ be pairs of functions and metrics on manifolds $M^\alpha,M^\beta,M^\gamma$, and suppose $h^{\gamma\beta}$ and $h^{\beta\alpha}$ are flow maps with respect to the gradient flows. Let $N^{\gamma}$ be an isolating neighborhood of the gradient flow of $(f^\gamma,g^\gamma)$. Set
$$
N^{\beta}:=(h^{\gamma\beta})^{-1}(N^\gamma),\quad N^\alpha=(h^{\beta\alpha})^{-1}(N^\alpha). 
$$
Then these are isolating neighborhoods, the maps $h^{\gamma\beta}$, $h^{\beta\alpha}$ and $h^{\gamma\beta}\circ h^{\beta\alpha}$ are isolated and the following diagram commutes
$$
\xymatrix{HM_*(f^\alpha,g^\alpha,N^\alpha)\ar[rd]_{(h^{\gamma\beta}\circ h^{\beta\alpha})_*}\ar[r]^{h^{\beta\alpha}_*}&HM_*(f^\beta,g^\beta,N^\beta)\ar[d]^{h^{\gamma\beta}_*}\\
&HM_*(f^\gamma,g^\gamma,N^\gamma)}
$$
\label{prop:degeneratefunctoriality}
\end{proposition}
\begin{proof}
The proposition follows from combining Propositions~\ref{prop:flowmaps},~\ref{prop:localmorsefunctorial} and ~\ref{prop:degenerateinduced}.
\end{proof}

Theorem~\ref{thm:localmorsefunctoriality} now follows from the fact that isolated homotopic maps induce the same maps in local Morse homology.  

\section{Functoriality in Morse-Conley-Floer homology}
\label{sec:functorialitymorseconleyfloer}

We use the induced maps of local Morse homology to define induced maps for flow maps in Morse-Conley-Floer homology. 

\begin{theorem}
Let $h^{\beta\alpha}:M^\alpha\rightarrow M^\beta$ be a flow map. Let $S^\beta\subset M^\beta$ be an isolated invariant set. Then $S^\alpha=(h^{\beta\alpha})^{-1}(S^\beta)$ is an isolated invariant set and there is an induced map
$$
h^{\beta\alpha}_*:HI_*(S^\alpha,\phi^\alpha)\rightarrow HI_*(S^\beta,\phi^\beta),
$$
which is functorial: The identity is mapped to the identity and the diagram
$$
\xymatrix{HI_*((h^{\gamma\beta}\circ h^{\beta\alpha})^{-1}(S^\gamma),\phi^\alpha)\ar[r]^-{h^{\beta\alpha}_*}\ar[dr]_{(h^{\gamma\beta}\circ h^{\beta\alpha})_*}&HI_*((h^{\gamma\beta})^{-1}(S^\gamma),\phi^\beta)\ar[d]^{h^{\gamma\beta}_*}\\
&HI_*(S^\gamma,\phi^\gamma).}
$$
commutes.
\label{thm:morseconleyfloerfunctoriality}
\end{theorem}
\begin{proof}
The induced map is defined as follows. Let $f^\beta$ be a Lyapunov function\footnote{We previously denoted this by $f_{\phi^\beta}$ but this notation is too unwieldy here}  with respect to the isolating neighborhood $N^\beta$ of $(S^\beta,\phi^\beta)$. Then $N^\alpha:=(h^{\beta\alpha})^{-1}(N^\beta)$ is an isolating neighborhood of $S^\alpha=(h^{\beta\alpha})^{-1}(S^\beta)$, and $f^\alpha:=f^\beta\circ h^{\beta\alpha}$ is a Lyapunov function by Proposition~\ref{prop:lyapunovpullback}. Moreover $h^{\beta\alpha}$ is isolated with respect to the gradient flows for any two metrics $g^\alpha$ and $g^\beta$. By Proposition~\ref{prop:degenerateinduced} we have an induced map
$$
h^{\beta\alpha}_*:HM_*(f^\alpha,g^\alpha,N^\alpha)\rightarrow HM_*(f^\beta,g^\beta,N^\beta)
$$
computed by perturbing everything to a transverse situation, preserving isolation, and counting as in Proposition~\ref{prop:localmorseinduced}. 

We now want to prove that the induced map passes to the inverse limit. If $f^\delta$ is another Lyapunov function with respect to an isolating neighborhood $N^\delta$ of $(S^\beta,\phi^\beta)$, and $N^\gamma=(h^{\beta\alpha})^{-1}(N^\delta)$, $f^\gamma=f^\delta\circ h^{\beta\alpha}$ then we cannot directly compare the local Morse homologies. However, as in \cite[Theorem 4.7]{rotvandervorst} the sets $N^\beta\cap N^\delta$ and $N^\alpha\cap N^\gamma$ are isolating neighborhoods, and $h^{\beta\alpha}$ is isolated with respect to these since $S_{h^{\beta\alpha}}\subset S^\alpha$ and $h^{\beta\alpha}(S_{h^{\beta\alpha}})\subset S^\beta$. 

The claim is that the following diagram in local Morse homology commutes
$$
\xymatrix{HM_*(f^\alpha,g^\alpha,N^\alpha)\ar[r]^-{h^{\beta\alpha}_*}\ar[d]&HM_*(f^\beta,g^\beta,N^\beta)\ar[d]\\
HM_*(f^\alpha,g^\alpha,N^\alpha\cap N^\gamma)\ar[r]^-{h^{\beta\alpha}_*}\ar[d]_{\Phi^{\gamma\alpha}_*}&HM_*(f^\beta,g^\beta,N^\beta\cap N^\delta)\ar[d]^{\Phi^{\delta\beta}_*}\\
HM_*(f^\gamma,g^\gamma,N^\alpha\cap N^\gamma)\ar[r]^-{h^{\beta\alpha}_*}\ar[d]&HM_*(f^\delta,g^\delta,N^\beta\cap N^\delta)\ar[d]\\
HM_*(f^\gamma,g^\gamma,N^\gamma)\ar[r]^-{h^{\beta\alpha}_*}&HM_*(f^\delta,g^\delta,N^\delta)}
$$
where the vertical maps are all isomorphisms. This proves that we have a well defined map in Morse-Conley-Floer homology. 

The commutativity of the square in the middle is induced by continuation as in Proposition~\ref{prop:degenerateinduced}. Gradient flows of different Lyapunov functions on the same isolating neighborhood are isolated homotopic by the proof of \cite[Theorem 4.4]{rotvandervorst}. Moreover since the set  $S_{h^{\beta\alpha}}$ of Equation~\bref{eq:isolatedmapset} is contained in $S^\alpha$ and $h^{\beta\alpha}(S^\beta)\subset S^\beta$ for any choice of Lyapunov function. The isolated homotopies of the gradient flows preserve the isolation of $h^{\beta\alpha}$. 

The upper square in the diagram commutes because the set $S_{h^{\beta\alpha}}$ of Equation~\bref{eq:isolatedmapset} is contained in $S^\alpha$ and $h^{\beta\alpha}(S^\alpha)\subset S^\beta$. It is possible to perturb the function $f^\alpha$ to a Morse-Smale function preserving the isolation in $N^\alpha\cap N^\gamma$ and similarly for $f^\beta$, while also preserving isolation of $h^{\beta\alpha}$. We use the openness of isolated maps, cf.~Proposition~\ref{prop:isolatedmapsareopen} and genericity of transverse maps, Theorem~\ref{thm:genericity}. Then the counts with respect to $N^\alpha$ and $N^\alpha\cap N^\gamma$ and $N^\beta$ and $N^\beta\cap N^\delta$ are the same for such perturbations from which it follows that the diagram commutes. The situation for the lower square is completely analogous. 

We have a well defined map in Morse-Conley-Floer homology. Functoriality might not be clear at this point, however If $h^{\gamma\beta}:M^\beta\rightarrow M^\gamma$ is another flow map, and $f^\gamma$ is a Lyapunov function for $\phi^\gamma$, then $f^\beta=f^\gamma\circ h^{\gamma\beta}$ and $f^\alpha=f^\gamma\circ h^{\gamma\beta}\circ h^{\beta\alpha}$ are Lyapunov functions, for the obvious isolated invariant sets and neighborhoods. Choose auxiliary metrics $g^\alpha,g^\beta,g^\gamma$, and the denote the gradient flows of the Lyapunov functions by $\psi^\alpha,\psi^\beta,\psi^\gamma$. Define for $R\geq 0$ and $p\in S_{h^{\gamma\beta}\circ \psi^{\beta}_R\circ h^{\beta\alpha}}$ the map $b_{p,R}:\mR\rightarrow \mR$ by
$$
b_{p,R}(t)=\begin{cases}
f^\alpha(\psi^\alpha(t,p))\qquad\qquad\qquad\qquad\text{for }&t<0\\
f^\beta(\psi^\beta(t,h^{\beta\alpha}(p)))\qquad\qquad &0<t<R\\
f^\gamma(\psi^\gamma(t,h^{\gamma\beta}\circ \psi^\beta_R\circ h^{\beta\alpha}(p)))&t>R.
\end{cases}
$$
The map $b_{p,R}$ is continuous and smooth outside $t=0,R$, and outside $t=0,R$ we see that $\frac{d}{dt}b_{p,R}\leq 0$. As in Proposition~\ref{prop:lyapunovpullback} we have limits $\lim_{t\rightarrow\pm\infty} b_{p,R}(t)=c$, where $c$ is the constant with $f^\gamma\bigr|_{S^\gamma}\equiv c$. It follows that $ S_{h^{\gamma\beta}\circ \psi_R^\beta\circ h^{\beta\alpha}}\subset S^\alpha$ and hence that the orbits through $p\in S_{h^{\gamma\beta}\circ \psi_R^\beta\circ h^{\beta\alpha}}$ are contained in $\Int N^\alpha$ and the orbits through $h^{\gamma\beta}\circ\psi_R^\beta\circ h^{\beta\alpha}(p)$ are contained in $\Int N^\gamma$ for all $R$. Thus $h^{\gamma\beta}\circ \psi^{\beta}_R\circ h^{\beta\alpha}$ is isolated for all $R\geq 0$. Perturbing everything to a transverse situation as before preserving isolation we get from Proposition~\ref{prop:localmorsefunctorial} functoriality in Morse-Conley-Floer homology.
\end{proof}

\section{Duality}
\label{sec:duality}
Given a chain complex $(C_*,\partial_*)$ the dual complex is defined via $C^k:=\Hom(C_k,\mZ)$ with  boundary operator $\delta^k:C^{k}\rightarrow C^{k+1}$ the pullback of the boundary operator $\partial_{k+1}:C_{k+1}\rightarrow C_{k}$. Thus for $\eta\in C^{k}$ we have $\delta^k \eta=\eta\circ \partial_{k+1}:C_{k+1}\rightarrow \mZ$. The homology of the dual complex $C^*$ is the \emph{cohomology} of $C_*$ and is denoted by $H^k(C_*):=\ker \delta ^{k}/\im \delta^{k-1}.$ 
\subsection{Morse cohomology and Poincar\'e duality}

If $A,B$ are oriented and cooriented submanifolds of a manifold $M$ that intersect transversely, the intersection $A\cap B$ is an oriented submanifold. If the ambient manifold $M$ is oriented, the exact sequence $0\rightarrow TA\rightarrow TM\rightarrow NA\rightarrow 0$ coorients $A$, and similarly orients $B$. The orientation on $B\cap A$, seen as an oriented submanifold, is related to the orientation of $A\cap B$ by the formula
\begin{equation}
\label{eq:reversingorientations}
\Or(B\cap A)=(-1)^{\dim A\cap B+\dim A\dim B}\Or(A\cap B).
\end{equation}
Now let $\cQ^\alpha=\{M^\alpha,f^\alpha,g^\alpha,\co^\alpha\}$ be a Morse datum on an \emph{oriented closed} manifold $M^\alpha$. We have the exact sequence of vector spaces
$$
\xymatrix{0\ar[r]&T_xW^u(x)\ar[r]&T_xM\ar[r]&N_xW^u(x)\ar[r]&0.}
$$
Because $T_xW^u(x)$ is oriented by $\co^\alpha$ and $T_xM$ is also oriented, this sequence orients $N_xW^u(x)\cong T_xW^s(x)$. The stable manifolds are therefore also oriented. The stable manifolds of $(f^\alpha,g^\alpha)$ are precisely the \emph{unstable manifolds} of $(-f^\alpha,g^\alpha)$ which we orient by the above exact sequence. We denote this choice of orientation of the unstable manifolds of $(-f^\alpha,g^\alpha)$ by $\widehat\co^\alpha$
\begin{definition}
Let $\cQ^\alpha$ be a Morse datum on an oriented manifold. The \emph{dual Morse datum} $\widehat\cQ^\alpha$ is defined by $\widehat \cQ^\alpha=\{M^\alpha,-f^\alpha, g^\alpha,\widehat \co^\alpha\}.$
\end{definition}

Under our compactness assumptions the Morse complex is finitely generated. Hence the dual complex is finitely generated. A basis of $C^*(\cQ^\alpha)$ is the dual basis, given for $x\in \crit f^\alpha$, by
$$\eta_x(y)=\begin{cases}
1\qquad &x=y\\
0\qquad &x\not=y
\end{cases}
$$
Note that a critical point of $f^\alpha$ of index $|x|$ is a critical point of $-f^\alpha$ of index $m^\alpha-|x|$. Define the Poincar\'e duality map $\PD_k:C_k(\cQ^\alpha)\rightarrow C^{m^\alpha-k}(\widehat\cQ^\alpha)$ by
$$
\PD_k(x)=\eta_x.
$$
As sets, $W(x,y;\cQ^\alpha)$ equals $W(y,x;\widehat\cO^\alpha)$. Using Equation~\bref{eq:reversingorientations} we see that $\Or(W(x,y;\cQ^\alpha))=(-1)\Or(W(y,x;\widehat\cO^\alpha))$. The negative gradient flow of $\cQ^\alpha$ is \emph{minus} the negative gradient flow of $\widehat \cQ^\alpha$. Quotienting out the induced $\mR$ actions we see that see that the minus sign dissappears and that $n(x,y;\cQ^\alpha)=n(y,x;\widehat \cQ^\alpha)$. Then
\begin{align*}
\delta^{n-k} \PD_k(x)=\eta_x\partial_{n-k+1}=&\sum_{|y;\widehat\cQ^\alpha|=|x;\widehat\cQ^\alpha|+1}n(y,x;\widehat\cQ^\alpha)\eta_y\\=&\sum_{|x;\cQ^\alpha|=|y;\cQ^\alpha|+1} n(x,y;\cQ^\alpha)\eta_y=\PD_{k-1}\partial_{k} x.
\end{align*}
\begin{theorem}
Let $\cQ^\alpha$ be a Morse datum on an oriented closed manifold. The Poincar\'e duality map is a chain map hence induces a map
$$
\PD_k:HM_k(\cQ^\alpha)\rightarrow HM^{m^\alpha-k}(\widehat \cQ^\alpha).
$$
The duality map is an isomorphism, and commutes with the canonical isomorphisms in the following way. If $\Phi^{\beta\alpha}_*:HM_*(\cQ^\alpha)\rightarrow HM_*(\cQ^\beta)$ and $(\widehat\Phi^{\beta\alpha})^*:HM^*(\widehat\cQ^\beta)\rightarrow HM^*(\widehat\cQ^\alpha)$ are canonical isomorphisms then
$$
\PD_k\Phi^{\beta\alpha}_k=(\widehat\Phi^{\beta\alpha})^{n-k}\PD_k.
$$
\label{thm:poincaremorse}
\end{theorem}
Since the duality map commutes with the canonical isomorphisms, and the gradient flows of $\cQ^\alpha$ and $\widehat \cQ^\alpha$ are isolated homotopic, this gives duality of the Morse complex of the manifold, i.e. there is a Poincar\'e duality map $\PD_k:HM_k(M^\alpha)\rightarrow HM^{m^\alpha-k}(M^\alpha)$.

\subsection{Duality in local Morse homology}
Recall that a closed subset $C$ of a manifold $M$ is \emph{orientable} if there exist a continuous section in the orientation bundle over $M$, cf.~\cite[Chapter VI.7]{Bredon_Topology}. Let $\cP^\alpha$ be a local Morse datum. The local Morse datum is \emph{orientable} if $S^\alpha=\Inv(N^\alpha,\psi^\alpha)$ is orientable. A choice of section of the orientation bundle is an \emph{orientation} of $\cP^\alpha$. If $M^\alpha$ is an oriented manifold, all closed subsets are oriented, thus on an orientable manifold all local Morse data are orientable. 

If a local Morse datum $\cP^\alpha$ is oriented we can define the dual local Morse datum $\widehat \cP^\alpha=\{M^\alpha, -f^\alpha,g^\alpha,N^\alpha,\widehat \co^\alpha\}$ as before. Again we have Poincar\'e duality isomorphisms $\PD_k:HM_k(\cP^\alpha)\rightarrow HM^{m^\alpha-k}(\widehat \cP^\alpha)$. A crucial difference is now that the gradient flow of $\cP^\alpha$ is in general \emph{not} isolated homotopic to the gradient flow of $\cP^\alpha$, thus $HM^{m^\alpha-k}(\widehat \cQ^\alpha)\not\cong HM^{m^\alpha-k}(\cQ^\alpha)$. % Recall that if $(f,g,N)$ is a triple of a function, metric and isolating neighborhood of the flow we define its local Morse homology via $HM_*(f,g,N)=\varprojlim HM_*(\cP^\alpha)$ over all local Morse data that are isolated homotopic to the gradient flow of $(f,g)$ in $N$. It is readily seen that $N$ also isolates the gradient flow of $(-f,g)$. We have the following Poincar\'e duality statement

\begin{theorem}
Let $(f,g,N)$ be a triple of a function a metric and an isolating neighborhood of the gradient flow. Assume that  $S=\Inv(N,\psi)$ is oriented. Then there are Poincar\'e duality isomorphisms
$$
PD_k:HM_k(f,g,N)\rightarrow HM^{m^\alpha-k}(-f,g,N).
$$
\label{thm:poincarelocalmorse}
\end{theorem}

\subsection{Duality in Morse-Conley-Floer homology}

If $\phi$ is a flow, then the \emph{reverse flow} $\phi^{-1}$ defined
via $\phi^{-1}(t,x)=\phi(-t,x)$ is also a flow. The following duality
statement is analogous to a theorem for the Conley index, due to
McCord~\cite{McCord:1992vy}.

\begin{theorem}
Let $S$ be an oriented isolated invariant set of a flow $\phi$. Then
there are Poincar\'e duality isomorphisms
$$
PD_k:HI_k(S,\phi)\rightarrow HI^{m-k}(S,\phi^{-1}).
$$
\end{theorem}
\begin{proof}
If $f_\phi$ is a Lyapunov function for $(S,\phi)$ then $-f_\phi$ is a
Lyapunov function for $(S,\phi^{-1})$. Let $N$ be an isolating
neighborhood. Choosing a metric $g$, we get by Theorem~\ref{thm:poincarelocalmorse} an isomorphism
$HM_k(f_\phi,g,N)\rightarrow HM^{n-k}(-f_\phi,g,N)$. The Poincar\'e
isomorphisms commute with the continuation isomorphisms hence combine
to a Poincar\'e duality isomorphism as in the theorem.
\end{proof}

% Note that an isolated invariant set on an orientable manifold is always orientable. Similarly a manifold is always orientable with $\mZ/2\mZ$ coefficients, hence for Morse-Conley-Floer homology with $\mZ/2\mZ$ coefficients the above duality statement always holds.

\begin{remark}
There exist various algebraic structures on Morse-Conley-Floer homology. A Morse homological interpretation of the cross product in cohomology, along with functoriality, allows one to define cup products. Functoriality and duality allows one to define shriek maps which give rise to intersection products. Details are available in~\cite{Rot:ww}. Cap products are more intrincate. There do not exist cap products
$$
\frown : HI^k(S,\phi)\otimes HI_l(S,\phi)\rightarrow HI_{k-l}(S,\phi)
$$
However we can define cap products
$$
\frown : HI^k(S,\phi)\otimes HM_l(f,g,N)\rightarrow HI_{k-l}(S,\phi^{-1}),
$$
Where $(f,g)$ satisfies the following properties. $N$ is an isolating neighborhood of $S$ as well as the gradient flow of $f$. Moreover $f>0$ on $N$ and $f\bigr|_{\partial N}=0$.
\end{remark}

\appendix

\section{Transverse maps are generic}
\label{sec:generic}

\begin{theorem}
Let $\cQ^\alpha,\cQ^\beta$ be Morse-Smale triples. The set $\scrT(\cQ^\alpha,\cQ^\beta)$ is residual in the compact-open topology, i.e. it contains a countable intersection of open and dense sets.
\label{thm:genericity}
\end{theorem}
\begin{proof}
Let $x\in\crit f^\alpha$ and $y\in \crit f^\beta$. We show that $$\scrT(x,y)=\{h\in C^\infty(M^\alpha,M^\beta)\,|\, h\Bigr|_{W^u(x)}\pitchfork W^s(y)\}$$ is residual, from which it follows that the set of transverse maps is residual. 

We first show density of $\scrT(x,y)$. The set 
$$\pitchfork(W^u(x),M^\beta; W^s(y))=\{h\in C^{\infty}(W^u(x),M^\beta)\,|\, h\pitchfork W^s(y)\},$$
 is residual, cf.~\cite[Theorem 2.1(a)]{Hirsch}, and by Baire's category theorem it is dense. We show that $h^{\beta\alpha}\in \pitchfork(M^\alpha,M^\beta;W^s(y))$ can be approximated in the compact-open topology by maps in $\scrT(x,y)$. Because all maps in $C^\infty(M^\alpha,M^\beta)$ can be approximated by such maps $h^{\beta\alpha}$ we get the required density. 

The unstable manifold $W^u(x)$ is contractible thus its normal bundle is contractible. We identify a neighborhood of $W^u(x)$ in $M^\alpha$ with $W^u(x)\times \mR^{n-|x|}$. By Parametric Transversality~\cite[Theorem 2.7 page 79]{Hirsch}, the set of $v\in \mR^{n-|x|}$ such that $h^{\beta\alpha}\bigr|_{W^u(x)\times \{v\}}\pitchfork W^s(y)$ is dense. For a given $v$ it is now possible to construct a flow whose time-$1$ map $\phi^v$ locally translates $W^u(x)\times\{0\}$ to $W^s(x)\times \{v\}$. Then $h^{\beta\alpha}$ can be approximated by $h^{\beta\alpha}\circ \phi^{v_k}$ with $v_k\rightarrow 0$. Thus $\scrT(x,y)$ is dense in $C^\infty(M^\alpha,M^\beta)$. 

We now argue that $\scrT(x,y)$ contains a countable intersection of open sets, from which it follows that this set is residual. Consider the restriction mapping $\rho_{W^u}:C^\infty(M^\alpha,M^\beta)\rightarrow C^\infty (W^u(x),M^\beta)$, which is continuous in both the weak and strong topology. Again the transversality theorem gives that the set of transverse maps $\pitchfork(W^u(x),M^\beta; W^s(y))$ is residual, that is $\pitchfork(W^u(x),M^\beta; W^s(y))\supset \bigcap_{k\in \mN}U_k$ with $U_k$ open and dense. Note that 
$$\scrT(x,y)=\rho^{-1}_{W^u(x)}\left(\pitchfork(W^u(x),M^\beta; W^s(y))\right)\supset \bigcap_{k\in \mN}\rho^{-1}_{W^u(x)}(U_k)$$ 
hence contains a countable intersection of open sets. By the previous reasoning $\scrT(x,y)$ is dense, hence the open sets must be dense as well, and $\scrT(x,y)$ is residual. 

Since there are only a countable number of critical points of $f^\alpha$ and $f^\beta$ it follows that $\scrT(\cQ^\alpha,\cQ^\beta)=\bigcap_{x\in \crit f^\alpha, y\in \crit f^\beta} \scrT(x,y)$ is residual.  
\end{proof}

\bibliographystyle{abbrv}
\bibliography{functoriality}

\end{sloppypar}
\end{document}